\documentclass[11pt,a4paper]{article}

\usepackage[a4paper, total={5.8in, 9in}]{geometry}

\usepackage{amsmath,amsfonts,amsthm,amssymb,hyperref}
\usepackage[english]{babel}
\usepackage{enumitem, tikz}
\usepackage{mdframed}

\newtheorem{theorem}{Theorem}
\newtheorem{lemma}[theorem]{Lemma}
\newtheorem{conjecture}[theorem]{Conjecture}

\newtheorem{proposition}[theorem]{Proposition}

\newtheorem{claim}[theorem]{Claim}
\newtheorem{problem}[theorem]{Problem}
\numberwithin{theorem}{section}
\numberwithin{equation}{section}

\usepackage{xcolor} 
\hypersetup{
    colorlinks,
    linkcolor={blue},
    citecolor={red},
    urlcolor={blue}
}

\newenvironment{cproof}
{\begin{proof}
 [Proof.]
 \vspace{-1.5\parsep}
}
{ \end{proof}}

\newcommand{\MM}{\overline{M}}

\newcommand{\Aa}{\overline{A}}
\newcommand{\ff}{\gamma}
\newcommand{\HH}{\overline{H}}
\newcommand{\CC}{\mathcal{C}}

\begin{document}

\title{
Forcing clique immersions through chromatic number
\thanks{This work supported by the European Research Council under the European Union's Seventh Framework Programme (FP7/2007-2013)/ERC Grant Agreement no. 279558.}}

\author{
Gregory Gauthier
\thanks{ Princeton University, Princeton, NJ, USA, \texttt{gjg2@math.princeton.edu}}, 
Tien-Nam Le
\thanks{Laboratoire d'Informatique du Parall\'elisme,
\'Ecole Normale Sup\'erieure de Lyon, France, \texttt{tien-nam.le@ens-lyon.fr}}, \\ and 
Paul Wollan
\thanks{Department of Computer Science, University of Rome, ``La Sapienza'',
Rome, Italy, \texttt{wollan@di.uniroma1.it}}}

\date{}

\maketitle

\begin{abstract}
Building on recent work of Dvo\v{r}\'ak and Yepremyan, we show that every simple graph of minimum degree $7t+7$ contains $K_t$ as an immersion and that every graph with chromatic number at least $3.54t + 4$ contains $K_t$ as an immersion. We also show that every graph on $n$ vertices with no stable set of size three contains $K_{2\lfloor n/5 \rfloor}$ as an immersion.
\end{abstract}


\section{Introduction} \label{section:introduction}

\subsection{Hadwiger's conjecture}
The graphs in this paper are simple and finite, while multigraphs may have loops and multiple edges.
A fundamental question in graph theory is the relationship between the chromatic number of a graph $G$ and the presence of certain structures in $G$.  One of the most well-known specific example of this type of question is the Four Color Theorem, which states that every planar graph is 4-colorable.
Hadwiger \cite{Had} in 1943 proposed a far-reaching generalization of the Four Color Theorem, which asserts that for all positive integers $t$, every graph of chromatic number $t$ contains $K_t$, the clique on $t$ vertices, as a minor. In 1937, Wagner \cite{wa} proved that the Hadwiger's conjecture for $t=5$ is equivalent to the Four Color Theorem. Robertson, Seymour, and Thomas \cite{RST} settle the conjecture for $t=6$, while the conjecture is still open for $t\ge 7$.
On the other hand, it was independently proved in 1984 by Kostochka and Thomasson \cite{Ko,Tho} that a graph without a $K_t$-minor is $O(k\sqrt{\log k})$-colorable for every $k\ge 1$, and there has been no improvement in the order $k\sqrt{\log k}$ since then. 

For graphs with no stable set of size three (i.e. there do not exist three vertices, all pairwise nonadjacent), Duchet and Meyniel \cite{DM} proposed an analogous conjecture to the Hadwiger's conjecture that every graph with $n$ vertices and no stable set of size three contains a $K_{\lceil n/2\rceil}$-minor and proved that such graphs contain $K_{\lceil n/3\rceil}$ as a minor, which remains the best bound to date. Plumber, Stiebitz, and Toft \cite{PST} showed that the conjecture of Duchet and Meyniel is indeed equivalent to the Hadwiger's conjecture for graphs with no stable set of size three. 

\subsection{Graph immersion}

In this paper, we focus on the immersion relation on graphs, which is a variant of minor relation (see \cite{RS}). 
We follow the definitions in \cite{Wo}.
Given loopless multigraphs $G,H$, we say that $G$ admits an \emph{immersion} of $H$ if there exists functions $\pi_1:V(H)\to V(G)$ and $\pi_2$ mapping the edges of $H$ to paths of $G$ satisfying the following:
\begin{itemize}
\item the map $\pi_1$ is an injection;
\item for every edge $e\in E(H)$ with endpoints $x$ and $y$, $\pi_2(e)$ is a path with endpoints equal to $\pi_1(x)$ and $\pi_2(y)$; and
\item for edges $e,e'\in E(H)$, $e\ne e'$, $\pi_2(e)$ and $\pi_2(e')$ have no edge in common.
\end{itemize}
We say that $G$ admits a \emph{strong immersion} of $H$ if the following condition holds as well.
\begin{itemize}
\item For every edge $e\in E(H)$ with endpoints $x$ and $y$, the path $\pi_2(e)$ intersects the set $\pi_1(V(H))$ only in its endpoints.
\end{itemize}
The vertices $\{\pi_1(x) : x \in V (H)\}$ are the \emph{branch vertices} of the immersion. We will
also say that $G$ (strongly) immerses $H$ or alternatively that $G$ contains $H$ as a (strong) immersion.

We can alternately define immersions as follows. Let $e_1$ and $e_2$ be distinct edges in $G$ such that the endpoints of $e_1$ are $x, y$ and the endpoints of $e_2$ are $y, z$. To \emph{split off} the edges $e_1$ and $e_2$, we delete the edges $e_1$ and $e_2$ from $G$ and add
a new edge $e$ with endpoints $x$ and $z$ (note
that this might result in a multi-edge or a loop). Then $G$ contains $H$ as an immersion if and
only if $H$ can be obtained from a subgraph of $G$ by repeatedly splitting off pairs of edges and deleting isolated vertices.

We consider a variant of Hadwiger's conjecture to {graph immersions} due to Lescure and Meynial \cite{LM} in 1989 and, independently, to Abu-Khzam and Langston \cite{AL} in 2003. 
The conjecture explicitly states the following.
\begin{conjecture}[\cite{AL}, \cite{LM}]\label{conj:main}
For every positive integer $t$, every graph with no $K_t$ immersion is properly colorable with at most $t-1$ colors.
\end{conjecture}

Conjecture \ref{conj:main} is trivial for $t \le 4$, and was independently proved by
Lescure and Meyniel \cite{LM} and DeVos et al. \cite{DKMO10} for $5 \le  t \le 7$.
One can immediately show that a minimum counterexample to Conjecture \ref{conj:main} has minimum degree $t-1$.  Thus, the conjecture provides additional motivation for the natural question of what is the smallest minimum degree necessary to force a clique immersion.  DeVos et al. \cite{mohar} showed that minimum degree $200t$ suffices to force a $K_t$ immersion in a simple graph.  This implies that every graph without a $K_t$-immersion is $200t$-colorable, providing the first linear bound for Conjecture \ref{conj:main}, while, as we discussed above, the best known bound for the Hadwiger's conjecture is superlinear.
The bound $200t$ was recently improved by Dvo\v{r}\'ak and Yepremyan \cite{dvo} to $11t+7$.

\begin{theorem}[Dvo\v{r}\'ak--Yepremyan, \cite{dvo}]\label{theorem:mindeg11}
Every graph with minimum degree at least $11t+7$ contains an immersion of $K_t$.
\end{theorem}

We give a new result on clique immersions in dense graphs; we leave the exact statement for Section \ref{section:dense} below.  As a consequence, it is possible to improve the analysis in \cite{dvo} and obtain the following bound.

\begin{theorem}\label{theorem:mindeg}
Every graph with minimum degree at least $7t+7$ contains an immersion of $K_t$.
\end{theorem}

Conjecture \ref{conj:main} can be relaxed to consider the following question.
\begin{problem}
What is the smallest function $f$ such that for all positive $t$ and all graphs $G$ with $\chi(G) \ge f(t)$, it holds that $G$ contains $K_t$ as an immersion.
\end{problem}
As observed above, a minimum counterexample to Conjecture \ref{conj:main} has minimum degree $t-1$.  Thus by Theorem \ref{theorem:mindeg}, we get that chromatic number at least $f(t) = 7t+8$ forces a $K_t$ immersion.  By combining our results for dense graphs with arguments based on analyzing Kempe chains in proper colorings of graphs, we obtain the following improved bound.

\begin{theorem}\label{theorem:chromatic}
Every graph with chromatic number at least $3.54t+4$ contains an immersion of $K_t$.
\end{theorem}

For graphs with no stable set of size three, Vegara \cite{V17} proposed a similar conjecture as that of Duchet and Meyniel that every graph with $n$ vertices and no stable set of size three contains a strong $K_{\lceil n/2\rceil}$-immersion and proved that it is equivalent to Conjecture \ref{conj:main} for graphs with no stable set of size three. In the same paper, Vegara showed that a relaxation to $K_{\lceil n/3\rceil}$-immersion holds. We improve this to $K_{2\lfloor n/5 \rfloor}$.

\begin{theorem}\label{theorem:2.5}
For every integer $n\ge 1$, every graph $G$ with $n$ vertices and no stable set of size three has a strong immersion of $K_{2\lfloor n/5 \rfloor}$.
\end{theorem}

An extended abstract presenting Theorems \ref{theorem:mindeg} and \ref{theorem:chromatic} appeared in 2016 \cite{LW}.
\subsection{Notation}

Given a multigraph $G$ and distinct vertices $u,v\in V(G)$, if there are $k\ge 2$ edges between $u$ and $v$, we say that $uv$ is a \emph{multi-edge} with \emph{multiplicity} $k$, and if $u$ is not adjacent to $v$, we say that $uv$  is a \emph{missing edge}.
We denote by $N_G(v)$ the (non-repeated) set of neighbors of $v$ in $G$, and by $d_G(v)$ the {degree} of $v$ in $G$ (where a loop is counted $2$ and a multi-edge with multiplicity $k$ is counted $k$). 
We denote by $E_G(v)$ the multi-set of edges (loops are excluded) incident with $v$ (if $uv$ is a multi-edge of multicity $k$ then there are $k$ edges $uv$ in $E_G(v)$).
Given $X\subseteq V(G)$, we denote by $f_G(v|X)$ the number of vertices in $X\backslash\{v\}$ which are not adjacent to $v$ in $G$, and we write $f_G(v)=f_G(v|V(G))$ for short.
When it is clear in the context, we omit the subscript $G$ in this notation. Note that if $G$ is simple, then $d(v)=|N(v)|=|E(v)|=|V(G)|-f(v)-1$, but may not be the case if $G$ is a multigraph.

Given a multigraph $G$ and a subset $M$ of $V(G)$, let $G[M]$ denote the subgraph of $G$ induced by $M$.
Given a path linking vertices $u$ and $v$, to \emph{split off the path}, we delete the edges of the path and add an edge $uv$ to $G$. Given a vertex $v$ with $|E_G(v)|$ even, to \emph{suppress} $v$, we first match all edges of $E_G(v)$ into pairs; then we split off every pair $\{vu,vw\}$ of the matching, and finally delete $v$ and its loops (if any). 
Note that after suppressing a vertex, the degree of other vertices are unchanged. Both operations (splitting off a path and suppressing a vertex) can be expressed as a sequence of splitting off pairs of edges.
Given two multigraphs $G$ and $G'$, we define the \emph{union} of $G$ and $G'$, denoted $G \cup G'=G^*$ to be the multigraph with vertex set $V(G) \cup V(G')$ and the following edge set. For every two vertices $u$ and $v$ in $V(G) \cup V(G')$, the number of edges $uv$ in $G^*$ is equal to the sum of the number of edges $uv$ in $G$ and $G'$.

The structure of the paper is as follows. In sections \ref{section:dense}, we give some results on clique immersion in dense graphs, which are necessary for the proofs of Theorems \ref{theorem:mindeg} and \ref{theorem:chromatic}. Then we prove Theorems  \ref{theorem:mindeg}, \ref{theorem:chromatic}, and \ref{theorem:2.5} in Sections \ref{section:minimum}, \ref{section:chromatic}, and \ref{section:2.5}, respectively.

\section{Clique immersion in dense graphs} \label{section:dense}

In the following lemma, we show that if $G$ contains a set $M$ of $t$ vertices where the total sum of ``missing degree'' is small, then $G$ immerses a $K_t$ on $M$.

\begin{lemma} \label{lemma:average1}
Let $G=(V,E)$ be a graph with $n$ vertices and $M$ be a subset of $V$ with $t$ vertices.
If 
\begin{equation}\label{eq:dense1}
\sum_{v\in M} f_G(v)\le \Big(n-t-\max_{v\in M}f_G(v)\Big)t,
\end{equation}
then $G$ contains an immersion of $K_t$.
\end{lemma}

\begin{proof}
Let $\MM=V\backslash M$ and let $b=\max_{v\in M}f_{G}(v)$. 
Suppose that there are distinct vertices $v,v'\in M$ and $w\in \MM$ such that $vv'\notin E(G)$ and $vw,wv'\in E(G)$. By splitting off the path $vwv'$, we obtain the edge $vv'$ while $f(v)$ and $f(v')$ are unchanged, and so (\ref{eq:dense1}) still holds for the new graph. Thus by repeatedly finding such triples and splitting off, we obtain new graphs satisfying (\ref{eq:dense1}) while the number of edges strictly decreases after each step. 
Therefore the process must halt and return a graph $G_1=(V,E_1)$ satisfying
\begin{equation}\label{en:dense1}
\sum_{v\in M} f_{G_1}(v)\le \big(n-t-b\big)t, \text{ and}
\end{equation}
\begin{enumerate}[label=(\roman*)]
\item \label{en:dense2} there are no $v,v'\in M$ and $w\in \MM$ such that $vv'\notin E_1$ and $vw,wv'\in E_1$.
\end{enumerate}

For the rest of the proof, we write $f$ instead of $f_{G_1}$.  Let $r$ be the number of missing edges of $G_1$ with two endpoints in $M$, and $X$ be the set of endpoints of these missing edges. If $r=0$, then $G_1[M]$ is a copy of $K_t$, which proves the lemma. Hence we may suppose that $r\ge 1$.

For every $v\in X$, there is $v'\in M$ such that $vv'\notin E_1$. From \ref{en:dense2} we have
$f(v|\MM)+f(v'|\MM)\ge |\MM|=n-t;$
otherwise, there exists $w\in \MM$ such that $vw,wv'\in E_1$. 
Hence 
$$n-t\le f(v|\MM)+f(v'|\MM)\le f(v|\MM)+f(v')\le f(v|\MM)+b,$$ and so $f(v|\MM)\ge n-t-b$ for every $v\in X$. This gives 
\begin{equation}
\sum_{v\in X} f(v)= \sum_{v\in X}f(v|\MM)+\sum_{v\in X} f(v|M)\ge (n-t-b)|X| + 2r.\label{equation:XY}
\end{equation}

We will construct a $K_t$ immersion in $G_1$ as follows: for every non-adjacent pair of vertices $v,v'$ in $X$, we will obtain the edge $vv'$ by splitting off path $vwuw'v'$ for some $u\in Y= M \setminus X$ and $w,w' \in \MM$. 
As a first step to finding such 4-edge paths, for all $u \in Y$, define $$h(u)=\max\Big(0,\Big\lfloor\frac{n-t-b-f(u)+1}{2}\Big\rfloor\Big).$$
It holds that $2h(u)\ge n-t-b-f(u)$. Hence 
$$2\sum_{u\in Y}h(u) \ge (n-t-b)|Y|-\sum_{u\in Y}f(u).$$
Combining with \eqref{equation:XY}, and then with (\ref{en:dense1}) yields
\begin{align*}
2\sum_{u\in Y}h(u) -2r&\ge \Big((n-t-b)|Y|-\sum_{u\in Y}f(u)\Big)+\Big( (n-t-b)|X|-\sum_{v\in X}f(v)\Big)\\
&\ge (n-t-b)(|X|+|Y|)-\sum_{v\in M}f(v)\\
&\ge (n-t-b)t-(n-b-t)t= 0.
\end{align*}
Hence $\sum_{u\in Y}h(u) \ge r$. 

Choose arbitrarily two non-adjacent vertices $v,v'$ in $M$ (clearly $v,v'\in X$), and an arbitrary vertex $u\in Y$ such that $h(u)\ge 1$.
Such a vertex $u$ always exists as $\sum_{u\in Y}h(u)\ge r\ge 1$ and $h(u)$ is an integer for every $u$. 
By definition of function $h$, we have
$$f(u|\MM)\le f(u)\le n-t-b+1-2h(u)\le n-t-b-1.$$ 
From $f(v)\le b$, we have $$f(u|\MM)+f(v|\MM)\le (n-t-b-1)+f(v)\le n-t-1<|\MM|,$$ 
so $u$ and $v$ have a common neighbor $w\in \MM$. 
Similarly $u$ and $v'$ have a common neighbor $w'\in \MM$. If $w=w'$ then $vw,wv'\in E_1$, contrary to \ref{en:dense2}. 

By splitting off the path $vwuw'v'$, we get the edge $vv'$.  In doing so, 
we have that $f(v)$ and $f(v')$ remain unchanged while $f(u)$ increases by 2, i.e., $h(u)$ decreases by 1. 
Thus $\sum_{u\in Y}h(u)$ decreases by 1. 
However, the number of missing edges in $G_1[M]$ also decreases by 1, so we still have that $\sum_{u\in Y}h(u)$ is at least the number of missing edges in $G_1[M]$. 
We repeat the process above until we link all pairs of non-adjacent vertices in $M$, and so obtain a complete graph on $M$. Thus $G_1$ contains an immersion of $K_t$, and consequently, $G$ contains $K_t$ as an immersion as well.  This proves the lemma.
\end{proof}

As a corollary of Lemma \ref{lemma:average1}, the following lemma provides a more general bound for clique immersion of a graph by its average ``missing degree''.

\begin{lemma}\label{lemma:average2}
Let $G$ be a graph on $n$ vertices, and let $\ff=\sum_{v\in V(G)}f_G(v)/n$ be the average ``missing degree" of $G$. 
If $\ff \le n/2$, then $G$ contains an immersion of $K_t$ where $t=\min\big(
\lfloor n/2 \rfloor,\lfloor n -2\ff\rfloor\big)$.
\end{lemma}

\begin{proof}
Let $M$ be a set of $t=\min\big(
\lfloor n/2 \rfloor,\lfloor n -2\ff\rfloor\big)$ vertices minimizing $\sum_{v\in M}f(v)$. Let $b=\max_{v\in M}f(v)$ and $\MM=V(G)\backslash M$.

If $2b\le n-t$, note that $f(v)\le b$ for every $v\in M$, and so $\sum_{v\in M}f(v)\le bt\le (n-t-b)t$, and we apply Lemma \ref{lemma:average1} to complete the proof. 

Otherwise, $2b> n-t$. By the minimality of $f$ on $M$,  we have $f(w)\ge b$ for every $w\in \MM$. Hence
\begin{equation}
\sum_{v\in M}f(v)= \sum_{v\in V(G)}f(v)- \sum_{w\in \MM}f(w) \le \ff n-b(n-t).\label{equation:0}
\end{equation}

We now show that $\ff n-b(n-t) \le (n-t-b)t$.
Indeed,
\begin{align}
\ff n-b(n-t) &\le (n-t-b)t\nonumber \\
\Longleftrightarrow 2(\ff n-bn+bt) &\le 2(n-t-b)t\nonumber \\
\Longleftrightarrow\ \   2\ff n-n^2+tn &\le 2b(n-2t)-(n-t)(n-2t) \nonumber\\
\Longleftrightarrow\ \ \  (2\ff+t-n)n &\le (2b-n+t)(n-2t).
\label{equation:1}
\end{align}

Since $t=\min\big(
\lfloor n/2 \rfloor,\lfloor n -2\ff\rfloor\big)$, we have $2\ff \le n-t$ and $2t\le n$. Combining with $2b>n-t$ yields
$$(2\ff +t-n)n\le 0\le (2b-n+t)(n-2t).$$ Hence (\ref{equation:1}) holds, and so $\ff n-b(n-t) \le (n-t-b)t$.  This, together with equality (\ref{equation:0}), implies that $\sum_{v\in M}f(v)\le (n-t-b)t$, and we apply Lemma \ref{lemma:average1} to complete the proof. 
\end{proof}
In the case $n/4\le \ff\le n/2$, by tightening the analysis, we can slightly improve the bound in Lemma \ref{lemma:average2} to $t=\lfloor n-2\ff  \rfloor+1$, which is sharp even if $\ff$ is the maximum missing degree (see \cite{FW16}, Lemma 2.1). In the case $\ff<n/4$, the above technique could yield $t=\max\big(\lfloor n/2\rfloor,\lfloor n-\sqrt{2\ff n}\rfloor\big)$; however, $t=\lfloor n/2\rfloor$
is enough for our purpose.

\section{Forcing a clique immersion via minimum degree} \label{section:minimum}

In this section, we show how the proof of Theorem \ref{theorem:mindeg11} can be refined to give the proof of Theorem \ref{theorem:mindeg}. 
The main idea is as follows. Suppose, to reach a contradiction, that there is a graph with high minimum degree which does not contain a $K_t$-immersion. We choose such a graph $G$ with as few vertices as possible. If $G$ is dense, then we can find a $K_t$ immersion, a contradiction. Otherwise, $G$ is sparse, and so we can supress a vertex to get a smaller graph, which still has high minimum degree and does not contain a $K_t$-immersion, a contradiction again. The main difficulty is how to suppress a vertex of $G$ so that the new graph is still simple. 
We first state several results from \cite{dvo}.

\begin{proposition}[\cite{dvo}, Lemma 6]\label{lemma:completemul}
Every complete multipartite graph of minimum degree at least $t$ contains an immersion of $K_t$.
\end{proposition}

A graph on odd number of vertices is \emph{hypomatchable} if deleting any vertex results in a graph with a perfect matching.
\begin{proposition}[\cite{dvo}, Lemma 8]\label{lemma:edmonds} Fix $t$ and let $H$ be a graph not containing any complete multipartite subgraph with minimum degree at least $t$. 
Suppose that the complement graph $\overline{H}$ of $H$ neither has a perfect matching nor is hypomatchable. 
Then there exist disjoint subsets $W,L$ of $V(H)$ such that
\begin{itemize}
\item $|W|\le t-1$ and $|L|\ge |V(H)|-2|W|$;
\item $f_H(v)\le |W|$ for every $v\in W$; and
\item $uv\in E(H)$ for every $u\in W$ and $v\in L$.
\end{itemize}
\end{proposition}

Given a multigraph $G$, we say that a vertex $v$ of $G$ can be {well-suppressed} (in $G$) if we can suppress $v$ without creating any new loop or multi-edge in $G$. Precisely, $v$ can be \emph{well-suppressed} if there is a matching of edges of $E_G(v)$ such that 
\begin{itemize}
\item for every pair $\{vu_1,vu_2\}$, we have $u_1\ne u_2$ and $u_1u_2\notin E(G)$, and 
\item for every two pairs $\{vu_1,vu_2\}$ and $\{vu'_1,vu'_2\}$ we have $\{u_1,u_2\}\ne \{u_1',u_2'\}$.
\end{itemize}
A vertex $v$ can be \emph{nearly well-suppressed} if for all edges $e \in E_G(v)$, the vertex $v$ can be well-suppressed after deleting $e$.

Given a simple graph $G$, it is straightforward that if a vertex $v$ can be well-suppressed (nearly well-suppressed), then the complement graph of the induced subgraph $G[N(v)]$ has a perfect matching (is hypomatchable, respectively). The situation is more complex when $G$ is a multigraph. In the next lemma, we consider the case where some multi-edges are allowed.

\begin{lemma}\label{lem:hypo-inside}
Fix $t\ge 1$ and let $G'$ be a loopless multigraph with vertex set $V\cup \{z\}$ (where $z\notin V$) such that for every $v\in V$, $zv$ is either an edge or a multi-edge with multiplicity $2$.
Let $R$ be the set of vertices incident with $z$ by a multi-edge. If 
\begin{itemize}
\item $|V|-2|R|\ge 3t$,
\item $G:=G'[V]$ is simple and does not contains $K_t$ as an immersion, and
\item $z$ cannot be well-suppressed or nearly well-suppresed in $G'$, 

\end{itemize}
then there is a set $W\subseteq V$ such that $|W|\le t-1$ and $f_G(v)\le |W|+|R|$ for every $v\in W$.
\end{lemma}

\begin{proof}
We define an auxiliary (simple) graph $H$ as follows. 
Beginning with $G$, for every vertex $v\in R$, we add a \emph{clone} vertex $v_c$ to $H$ which has the following neighbors: all the vertices of $R$, all the neighbors of $v$ in $G$, and every other clone vertex $u_c$.  
Explicitly, $H$ has vertex set $V \cup \{v_c| v \in R\}$ and edge set 
$$E(H)=E(G) \cup \{u_cv| u,v \in R\}\cup \{u_cv_c| u, v \in R\} \cup \{v_cx| v\in R, vx\in E(G)\} .$$  
Each vertex in $H$ indeed corresponds to an edge of $E_{G'}(z)$, where each clone vertex $v_c$ represents the additional edge in the multi-edge $zv$. 

Let $\HH$ be the complement graph of $H$.
We will show that $\HH$ neither has a perfect matching nor is hypomatchable.
If $\HH$ has a perfect matching, then by the construction of $H$, that perfect matching corresponds to a matching of edges in $E_{G'}(z)$ such that 
\begin{itemize}
\item for every pair $\{zu_1,zu_2\}$, we have $u_1\ne u_2$ and $u_1u_2\notin E(G')$, and 
\item for every two pairs $\{zu_1,zu_2\}$ and $\{zu'_1,zu'_2\}$ we have $\{u_1,u_2\}\ne \{u_1',u_2'\}$.
\end{itemize}
Thus we can we can well-suppress $z$ in $G'$, a contradiction to the third assumption of the lemma.
If $\HH$ is hypomatchable, then for every $v\in V(H)$, there is a perfect matching of $V(H)\backslash \{v\}$ in $\HH$. The same argument shows that $z$ can be nearly well-suppressed in $G'$, a contradiction.
We conclude that $\HH$ neither has a perfect matching nor is it hypomatchable. 

Observe that removing a vertex of a complete multipartite graph with minimum degree $d$ results in a complete multipartite graph with minimum degree at least $d-1$. 
Hence suppose that $H$ contains a multipartite subgraph of minimum degree at least $|R|+t$.
By removing all clone vertices of $H$, we obtain $G$, which still contains a complete multipartite subgraph with minimum degree at least $(|R|+t)-|R|=t$. By Proposition \ref{lemma:completemul}, $G$ contains $K_t$ as an immersion, a contradiction. We conclude that $H$ does not contains any multipartite subgraph of minimum degree at least $|R|+t$.

Applying Proposition \ref{lemma:edmonds} to $H$, we obtain disjoint subsets $W',L'$ of $V(H)$ such that 
\begin{enumerate}[label=(\alph*)]
\item \label{en:3.1} $|W'|\le |R|+t-1$ and $|L'|\ge |V(H)|-2|W'|$;
\item \label{en:3.2} $f_H(v)\le |W'|$ for every $v\in W'$; and
\item \label{en:3.3} $uv\in E(H)$ for every $u\in W'$ and $v\in L'$.
\end{enumerate}
Let $R_c$ be the set of clone vertices of $H$ and $W=W'\backslash R_c$ and $L=L'\backslash R_c$. We will show that $W$ is a desired set.
By \ref{en:3.1} we have
$$|L'|\ge |V(H)|-2|W'|> (|V|+|R|)-2(|R|+t)\ge |V|-|R|-2t.$$
Thus $|L'|-|R|\ge|V|-2|R|-2t$.
Recall from the hypothesis that $|V|-2|R|\ge 3t$, and so $|L|\ge |L'|-|R|\ge t$. Note that by \ref{en:3.3}, $uv\in E(H)$ for every $u\in W$ and $v\in L$, and hence $uv\in E(G)$ for every $u\in W$ and $v\in L$. If $|W|\ge t$, then $G[W\cup L]$ contains a complete bipartite graph with minimum degree at least $t$, and so contains $K_t$ as an immersion by Proposition \ref{lemma:completemul}, a contradiction. Thus it holds that $|W|\le t-1$.

Note that $f_G(v)\le f_H(v)$ since $G$ is an induced subgraph of $H$. It follows from \ref{en:3.2} that $f_G(v)\le f_H(v)\le |W'|\le |W|+|R|$ for every $v\in W$. This completes the proof of the lemma.
\end{proof}

Given an integer $t>1$, we call a graph $t$-deficient if it can be obtained from a graph with minimum degree $t$ by removing a few edges. Precisely, a graph $G$ is \emph{$t$-deficient} if $\sum_{v\in V(G)}\max(0,t-d_{G}(v))<t$.

\begin{proposition}[\cite{dvo}, Lemma 13]\label{lemma:eulerian}
If $G$ is a graph of minimum degree at least $7t+7$ that does not contain an immersion of $K_t$, then $G$ contains an immersion of some $7t$-deficient eulerian graph $G'$.
\end{proposition}

\begin{proposition}[\cite{dvo}, Lemma 15]\label{lemma:eulerian2}
Every $7t$-deficient eulerian graph contains a vertex of degree at least $7t$.
\end{proposition}

The main technical step in the proof of Theorem \ref{theorem:mindeg} is the following lemma.  Dvo\v{r}\'ak and Yepremyan proved a similar result for $11t + 7$-deficient eulerian graphs in \cite{dvo}.
\begin{lemma}\label{lemma:mainmindeg}
Every $7t$-deficient eulerian graph contains an immersion of $K_t$.
\end{lemma}

Theorem \ref{theorem:mindeg} follows easily from Lemma \ref{lemma:mainmindeg}. 
Suppose for a contradiction that there exists a graph $G$ of minimum degree at least $7t+7$ and does not have an immersion of $K_t$. 
By Proposition \ref{lemma:eulerian}, $G$ contains an immersion of a $7t$-deficient eulerian graph $G'$. 
By Lemma \ref{lemma:mainmindeg}, $G'$ contains an immersion of $K_t$, a contradiction.

\begin{proof}[Proof of Lemma \ref{lemma:mainmindeg}]
Suppose that there exists a $7t$-deficient simple eulerian graph which does not contain an immersion of $K_t$. Let $G=(V,E)$ be such a graph with as few vertices as possible. The idea of the proof is as follows. If $G$ has few edges, we show it would be possible to well-suppress some vertex of $G$ to get a smaller counterexample, a contradiction. Hence $G$ has many edges. We are then able to find in $G$ two disjoint sets of vertices $A$ and $B$ of size around $t$ and $6t$, respectively, such that there are very few missing edges between $A$ and $B$.  We apply Lemma \ref{lemma:average1} to obtain an immersion of $K_t$ and so reach a contradiction.
 
Let $z_1$ be a vertex in $G$ with $d(z_1)\ge 7t$, as guaranteed by Proposition \ref{lemma:eulerian2}. 
Let $1\le p< t$ be the maximum integer such that there exists an ordered set $A=\{z_1,z_2,...,z_p\}$ satisfying
\begin{equation}
f(z_i|B)\le p+i+r_i,\  \text {for all }  i\ge 2. \label{equation:induction}
\end{equation}
where $B=N(z_1)\backslash A$ and $r_i=\big|\{j\le i: z_j\notin N(z_1)\}\big|$ for every $i\ge 2$. 
Such number $p$ clearly exists since \eqref{equation:induction} trivially holds for $A=\{z_1\}$.
Since $|N(z_1)\cap A|=p-r_p$, we have

\begin{equation}\label{eq:B}
|B|=|N(z_1)\backslash A|=d(z_1)-|N(z_1)\cap A|\ge 7t-p+r_p.
\end{equation}

Let $\Aa=V\backslash A$.
Starting with $G_p = G$, we will attempt to sequentially split off the vertices of $A$ in order $z_p, z_{p-1},\dots, z_1$ to create graphs $G_{p-1},G_{p-2}, \dots, G_0$. At each step, if we could find the complement of a perfect matching in $N_{G_i}(z_i)$, we could split off $z_i$ to obtain $G_{i-1}$ and maintaining the property that $G_{i-1}$ is simple.  However, the requirement that $N_{G_i}(z_i)$ have the complement of a perfect matching is too strong and so we will have to slightly relax it.  In doing so, we will need to introduce parallel edges into the graphs $G_i$, but we will want to do so in a tightly controlled manner.  This leads us to the following definition.

 Fix $q$, $0 \le q \le p$ and multigraphs $G_i$, $q \le i \le p$ which satisfy the following.
\begin{enumerate}[label=(\roman*)]
\item \label{en:e.1} $G_p =G$ and for all $i$, $q \le i < p$, $G_i$ is obtained from $G_{i+1}$ by suppressing $z_{i+1}$.
\item  \label{en:e.2} For all $i$, $q \le i \le p$, $G_i[\Aa]$ is simple.
\item  \label{en:e.3} For all $i$, every multi-edge of $G_i$ with an endpoint in $\Aa$ has multiplicity 2.
\item  \label{en:e.5}  For all $j,2\le j \le q$, there are at most $r_p - r_q$ multi-edges from $z_j$ to vertices of $\Aa$ in $G_q$, and there are at most $p - q$ multi-edges from $z_1$ to vertices of $\Aa$ in $G_q$,
\item  \label{en:e.4} There are at least $|\Aa|-p+q$ vertices in $\Aa$ not incident with any multi-edge in any $G_i,q \le i \le p$.

\item  \label{en:e.6} Given $v\in \Aa$ and $z\in A$, if $vz$ is a multi-edge in some $G_i,q\le i
\le p$, then for every $z'\in A, z'\ne z$ and every $j, q\le j \le p$, $vz'$ is not a multi-edge in $G_j$.

\item  \label{en:e.7}  Subject to \ref{en:e.1} -- \ref{en:e.6}, we choose $q$ and $G_i$, $q \le i \le p$ to minimize $q$.
\end{enumerate}
Such a number $q$ and multigraphs $G_i$, $q \le i \le p$ trivially exist, given the observation that $q = p$ and $G_p = G$ satisfy \ref{en:e.1} -- \ref{en:e.6} as $G$ is simple. 

We begin with the observation that $q>0$.  Otherwise, the graph $G_0$ does not contain $K_t$ as an immersion because $G_0$ itself immerses in $G$ by construction.  Moreover, $G_0$ is simple by \ref{en:e.2}, and for all $v \in V(G_0)$, $d_{G_0}(v) = d_{G}(v)$.  We conclude that $G_0$ is both eulerian and $t$-deficient, contrary to our choice of $G$ to be a counterexample on a minimum number of vertices.  

We now consider the graph $G_q$ and keep in mind that by the minimality of $q$ in \ref{en:e.7}, we cannot supress $z_q$ to obtain $G_{q-1}$ which satisfies all \ref{en:e.1} -- \ref{en:e.6}. 
Let $X=N_{G_{q}}(z_q)\cap \Aa$. We will show that $G':=G_q[X\cup\{z_q\}]$ satisfies all hypotheses of Lemma \ref{lem:hypo-inside}.
From \ref{en:e.2} and \ref{en:e.3}, we have $G'$ is a loopless multigraph with vertex set $X\cup \{z_q\}$ such that for every $v\in X$, $z_qv$ is either an edge or a multi-edge with multiplicity 2. Let $R$ be the set of vertices in $X$ incident with $z_q$ by a multi-edge. Then by \ref{en:e.5} we have
\begin{equation} \label{eq:R}
\left\{ \begin{array}{ll}
|R|\le p-1 \ \ \ \ \ \  \ \ \ \text{ if } q=1,\\
|R|\le r_p-r_q \ \ \ \ \ \ \ \text{ if } q>1. \end{array} \right. 
\end{equation}

\begin{claim}
$G'$ satisfies all hypotheses of Lemma \ref{lem:hypo-inside}.
\end{claim}
\begin{cproof}
We verify the hypotheses one by one.

\begin{itemize}
\item $G'[X]=G_q[X]$ is simple and does not contains $K_t$ as an immersion.
\end{itemize}
$G_q[X]$ is simple by \ref{en:e.2}, and does not contains $K_t$ as an immersion by \ref{en:e.1} and the assumption that $G$ does not contains $K_t$ as an immersion.

\begin{itemize}
\item $|X|-2|R|\ge 3t$.
\end{itemize}
To prove $|X|-2|R|\ge 3t$, note that $|B\backslash X|$ is the number of vertices in $B$ not adjacent to $z_q$ in $G_q$, which is at most the number of vertices in $B$ not adjacent to $z_q$ in $G$ since no edge between $z_q$ and $B$ have been removed in suppressing $z_p, \dots, z_{q+1}$. Thus $|B\backslash X|\le f_G(z_q|B)$. Combining with \eqref{equation:induction} we have
\begin{equation} \label{eq:BX}
\left\{ \begin{array}{ll}
|B\backslash X|=0 \ \ \ \ \ \ \ \ \ \ \ \ \ \ \text{ if } q=1,\\
|B\backslash X|\le p+q+r_q \ \ \ \text{ if } q>1. \end{array} \right. 
\end{equation}

In the case $q>1$, by \eqref{eq:B},
$$|X|\ge |B|-|B\backslash X|\ge (7t-p+r_p)-(p+q+r_q).$$
From the fact that $t\ge \max(p,q,r_p)$ and \eqref{eq:R}, we have
$$|X|-2|R|\ge 7t-2p-q-r_p+r_q\ge 3t.$$

In the case $q=1$, by \eqref{eq:B},
$$|X|\ge |B|-|B\backslash X|\ge |B|\ge 7t-p+r_p.$$ Hence from \eqref{eq:R} we have $|X|-2|R|\ge 7t-3p+r_p\ge 3t$.

\begin{itemize}
\item $z_q$ cannot be well-suppressed or nearly well-suppressed in $G'$.
\end{itemize}
Suppose that $z_q$ can be well-suppressed in $G'$. We first split off all edges from $z_q$ to $X$ by that matching. Then there are even number of edges incident with $z_q$ remaining in $G_q$, all from $z_q$ to $A$ since $X=N_{G_{q}}(z_q)\cap \Aa$. We now suppress $z_q$ in $G_q$ arbitrarily to obtain $G_{q-1}$. Since we do not create any new edge between $A$ and $\Aa$, \ref{en:e.1} -- \ref{en:e.6} hold trivially for $G_{q-1}$, which contradicts \ref{en:e.7}.

As the second case, suppose that $z_q$ can be nearly well-suppressed in $G'$. Pick a vertex $v\in X$ which is not incident with any multi-edge in $G_i$, for all $q \le i \le p$. 
Such vertex $v$ exists since by \ref{en:e.4}, there was at most $p-q$ distinct vertices of $\overline{A}$ incident with some multi-edge over all $G_i$, $q \le i \le p$, while $|X|\ge 3t>p-q$ (as we show above that $|X|-2|R|\ge 3t)$. 

Since $z_q$ is can be nearly well-suppressed in $G'$, if we remove the edge $z_qv$ in $G'$, we can well-suppress $z_q$ (in $G'$), and we do so. 
Since $d_{G_q}(z_q)$ is even, $z_q$ must be adjacent to some vertex $z_s$ with $s<q$. We choose such $s$ as small as possible and split off $z_sz_qv$. We now suppress $z_q$ in $G_q$ arbitrarily to obtain $G_{q-1}$ and will show that $G_{q-1}$ satisfies  \ref{en:e.1} -- \ref{en:e.6} and hence violates \ref{en:e.7}. 
Properties \ref{en:e.1} and \ref{en:e.2} hold trivially. The only possible new multi-edge that we have created is $z_sv$. Since $v$ is not incident with any multi-edge in $G_i$ for all $q \le i \le p$, \ref{en:e.3}, \ref{en:e.4} and \ref{en:e.6} hold for $G_{q-1}$. 
To prove \ref{en:e.5}, first observe that \ref{en:e.5} clearly holds if $z_s=z_1$. If $z_s\ne z_1$, then $z_1$ is not incident with $z_q$ by the choice of $s$, and so $r_{q-1}= r_{q}-1$ by the defintion of function $r$.
Thus $r_p - r_{q-1}= r_p - r_q +1$ and therefore \ref{en:e.5} holds. 
\end{cproof}

Hence $G'$ satisfies the hypotheses of Lemma \ref{lem:hypo-inside}, and so there is a set $W\subseteq X$ such that $|W|\le t-1$ and $f_{G_q[X]}(v)\le |W|+|R|$ for every $v\in W$. 

We next show that $|W|\ge t-p$. To do so, we need the following claim.
\begin{claim}\label{cl:W-main}
$f_G(v|B)\le |W|+2p+r_p$ for every $v\in W$.
\end{claim}
\begin{cproof}
We first show that $f_{G_q}(v|B)\le|W|+p+q+r_p.$
Note that $f_{G_q}(v|X)= f_{G_q[X]}(v)$ for every $v\in X$, and so
$$f_{G_q}(v|B) \le f_{G_q}(v|X)+f_{G_q}(v|B\backslash X)\le (|W|+|R|)+|B\backslash X|.$$ 

If $q>1$, recall that $|B\backslash X|\le p+q+r_q$ from \eqref{eq:BX} and $|R|\le r_p-r_q$ from \eqref{eq:R}. Hence we have 
$$f_{G_q}(v|B)\le (|W|+r_p-r_q)+(p+q+r_q)\le |W|+p+q+r_p.$$

If $q=1$, recall that $|B\backslash X|=0$ from \eqref{eq:BX} and $|R|\le p-1$ from \eqref{eq:R}. Thus we have 
$$f_{G_q}(v|B)\le |W|+p<|W|+p+q+r_p.$$

We conclude that $f_{G_q}(v|B)\le|W|+p+q+r_p$ in all cases. To complete the claim, it suffices to
show that $$f_G(v|B)\le f_{G_q}(v|B)+(p-q)$$ for every $v\in X$. Fix $v\in X$. By property \ref{en:e.6}, there exists a value $s$ such that $z_{i}v$ is not a multi-edge in $G_i$ for every $i\ne s,q\le i\le p$.
Thus for every $i\ne s, q< i\le p$, there is at most one edge $z_iv$ in $G_i$, and so when we supress $z_i$ in $G_i$ to obtain $G_{i-1}$, we add at most one edge between $v$ and $B$ into $G_{i-1}$.
If $s>q$, note that there are at most two edges $z_sv$ in $G_s$ by property \ref{en:e.3}. Hence when we supress $z_s$ in $G_s$ to obtain $G_{s-1}$, we add at most two edges between $v$ and $B$ into $G_{s-1}$. Thus  from $G=G_p$, when we supress $z_p,...,z_{q+1}$ to get $G_q$, we add in total at most $p-q-1+1=p-q$ edges from $v$ to $B$, and so $$f_G(v|B)=f_{G_p}(v|B)\le f_{G_q}(v|B)+(p-q).$$
This proves the claim.
\end{cproof}

\begin{claim}
$|W|\ge t-p$.
\end{claim}
\begin{cproof}
Suppose for a contradiction that $|W|+p=p^*<t$.  Let $A^*=A\cup W$ where elements in $W$ are enumerated $z_{p+1},...,z_{p^*}$, and let $B^*=N(z_1)\backslash A^*=B\backslash W.$ 
Then 
\begin{itemize}
\item $f_{G}(z_i|B^*) \le f_{G}(z_i|B) \le p+i+r_i\le p^*+i+r_i$ for every $i,2\le i\le p$. 
\item $f_{G}(z_i|B^*) \le f_{G}(z_j|B) \le |W|+2p+r_p\le p^*+i+r_i$ for every $i> p$ (note that $r_i\ge r_p$ since by definition $r$ is a non-decreasing function).
\end{itemize}
Hence \eqref{equation:induction} holds for $p^*$ and $A^*$, contrary to the maximality of $p$. Thus $|W|\ge t-p$.
\end{cproof}

Let $\hat{A}$ be an arbitrary set of $t-p$ vertices in $W$ and enumerate them $z_{p+1},...,z_t$. 
Let $M=A\cup \hat{A}$ and $\MM=B\backslash \hat{A}$. Let $U=M\cup \MM$ and $H=G[U]$. We will apply Lemma \ref{lemma:average1} to $H$ and deduce that $H$ must contain an immersion of $K_t$, which contradicts the assumption that $G$ does not contains an immersion of $K_t$ and so complete the proof of Lemma \ref{lemma:mainmindeg}. We first give some bounds for function $f$ in $H$. Observe that $f_H(z_i|\MM) = f_{G}(z_i|\MM) \le f_{G}(z_i|B)$ for every $i,1\le i\le p$. Note also that $f_{G}(z_1|B) =0$, and $f_{G}(z_i|B) \le 2p+i$ for every $i,1<i\le p$, and by Claim \ref{cl:W-main},
$$ f_{G}(z_i|B)\le |W|+2p+r_p\le t+2p+r_p$$ for every $i,p<i\le t$ (recall that $|W|\le t$). Thus 
\begin{equation} \label{eq:HH}
\left\{ \begin{array}{ll}
f_H(z_i|\MM)\le 2p+i \ \ \ \ \ \ \ \ \ \text{ if } i\le p,\\
f_H(z_i|\MM)\le t+2p+r_p \ \ \ \text{ if } i>p. \end{array} \right. 
\end{equation}
Also note that $|M|=t$, and from \eqref{eq:B}, $$|\MM|\ge |B|-|\hat{A}|\ge 7t-p+r_p-(t-p)=6t+r_p.$$

\begin{claim}\label{cl:HH}
$H$ contains an immersion of $K_t$.
\end{claim}
\begin{cproof}We consider two cases.

\textbf{Case 1:} $p\le t/2$. 
We have $f_H(z_i|M)\le |M|\le t$ for every $z_i$, and so
\begin{align*}
\sum_{z_i\in M}f_H(z_i)&\le \sum_{z_i\in M}f_H(z_i|M)+\sum_{z_i\in M}f_H(z_i|\MM)\\
&\le t^2+\sum_{1\le i\le p}f_H(z_i|\MM)+\sum_{p< i\le t}f_H(z_i|\MM)\\
&\le t^2+\sum_{i\le p}(2p+i)+\sum_{p< i\le t}(t+2p+r_p)\\
&\le t^2+3p^2+(t-p)(t+3p)\\
&\le 2t^2+2tp \le 3t^2.
\end{align*}
Since $2p\le t$, we have
$$\max_{z_i\in M}f_H(z_i)\le t+\max_{z_i\in M}f_H(z_i|\MM)\le t+(t+2p+r_p)\le 3t+r_p.$$
Note that $|U|=|M|+|\MM|=7t+r_p$. Hence 
$$\sum_{z_i\in M}f_H(z_i)\le 3t^2\le \Big(|U|-t-\max_{z_i\in M}f(z_i)\Big)t.$$
Apply Lemma \ref{lemma:average1} to obtain an immersion of $K_t$ on ${H}$.

\textbf{Case 2:} $p>t/2$. Set $q=|\hat{A}|=t-p$, and so $p>q$. The analysis of this case is more involved. Even though $\sum_{z_i\in M}f_H(z_i)$ is small, $\max_{z_i\in M}f_H(z_i)$ could be very large, and so we cannot apply Lemma \ref{lemma:average1} directly.
However, we can still use a similar argument to that in the proof of Lemma \ref{lemma:average1}.  We present the argument as an algorithm to explicitly find a series of splitting off of edges to yield a $K_t$ immersion by finding edge disjoint paths of length two or four linking the desired pairs of vertices.

Consider an arbitrary loopless multigraph $H'$ with vertex set $U$ and distinct vertices $z_i,z_j\in M$. 
We first define a subroutine called \textsc{Link$(H',z_i,z_j)$}: the algorithm finds $w\in \MM$ such that $z_iw,wz_j\in E(H')$ and then split off the path $z_iwz_j$ to obtain an edge $z_iz_j$.  The algorithm then returns $H'$ after splitting off the path.  Such a $w$ can be found by checking all possible choices for $w$.  In the case that multiple choices exist for $w$, the algorithm arbitrarily chooses one.  

In order to successfully run, the algorithm \textsc{Link$(H',z_i,z_j)$} assumes that the input satisfies:  
\begin{equation}
f_{H'}(z_i|\MM)+f_{H'}(z_j|\MM)< 6t+r_p\le |\MM|,
\label{equation:link}
\end{equation}
Under assumption (\ref{equation:link}), such a $w \in \MM$ must exist and therefore, the algorithm correctly terminates.  Note also that $z_i,z_j$ are adjacent after performing \textsc{Link$(H',z_i,z_j)$}, and that the input $H'$ contains the output graph as an immersion.

We now present the main algorithm to split off edges of $H$ to obtain a complete graph on $M=A\cup \hat{A}$. Set $H':=H$. The algorithm proceeds in stages.  In stage 1, we link all vertices between $\{z_{q+1},...,z_p\}$ and $\hat{A}$.  In stage 2, we link each pair of vertices between $\{z_{1},...,z_q\}$ and $\hat{A}$ with multi-edges of order two. Thus after stages 1 and 2, we obtain two edge-disjoint complete bipartite subgraphs, one between $A$ and $\hat{A}$ and another between $\{z_{1},...,z_q\}$ and $\hat{A}$ (the latter will be used later to obtain a complete graph on $\hat{A}$).  In stage 3, we link all vertices inside $A$, and then obtain a complete graph on $M$.

\medskip
\begin{mdframed}
{\sc Main}($H'$)
\begin{enumerate}

\item {Start with $s:=p$ and repeat the following whenever $s>q$.

\begin{enumerate}

\item[] Start with $i:=p+1$ and repeat the following whenever $i\le t$.

\textbf{\ \ \ \ \  } \textsc{Link$(H',z_s,z_i)$}, $i:=i+1$.

\item[] $s:=s-1$.

\end{enumerate}
}

\item {Start with $s:=q$ and repeat the following whenever $s\ge 1$.

\begin{enumerate}

\item[] Start with $i:=p+1$ and repeat the following whenever $i\le t$.

\textbf{\ \ \ \ \  } \textsc{Link$(H',z_s,z_i)$}, \textsc{Link$(H',z_s,z_i)$}, $i:=i+1$.

\item[] $s:=s-1$.

\end{enumerate}
}

\item {Start with $s:=p$ and repeat the following whenever $s\ge 1$.

\begin{enumerate}

\item[] Start with $i:=s-1$ and repeat the following whenever $i\ge 1$.

\textbf{\ \ \ \ \  } \textsc{Link$(H',z_s,z_i)$}, $i:=i-1$.

\item[] $s:=s-1$.

\end{enumerate}
}
\item Return $H'$.
\end{enumerate}

\end{mdframed}

\medskip

Suppose that we have performed \textsc{Main($H'$)} successfully. The output $H'$ contains two edge-disjoint complete bipartite subgraphs, $H_1$ from $A$ to $\hat{A}$, and $H_2$ from $\{z_{1},...,z_q\}$ to $\hat{A}$, and a complete graph $H_3$ on $A$.
We now show how to obtain from $H_2$ a complete graph $H_4$ on $\hat{A}$. 
Since $|\hat{A}|=q$, by Vizing Theorem, we can color the edges of an imagined complete graph on $\hat{A}$ by $q$ colors $\{1,2,...,q\}$ so that any two incident edges have different color. 
Now for every $z_i,z_j\in \hat{A}$, if the edge $z_iz_j$ in that imagined graph has color $s$, then we split off edges $z_iz_sz_j$ in the complete bipartite graph $H_2$ to get an edge $z_iz_j$, and so obtain a complete graph $H_4$ on $\hat{A}$. Hence $H_1\cup H_3\cup H_4$ is a complete graph on $M$. Thus the output $H'$ contains $K_t$ as an immersion, which implies that $H$ contains $K_t$ as an immersion.

It only remains to show that we can perform \textsc{Main($H'$)} successfully, which is equivalent to verifying that for each call to the subroutine \textsc{Link($H',z_i,z_j$)} we have that \eqref{equation:link} is satisfied.
We omit the subscript $H'$ of $f$ in the rest of this proof.
Observe that after performing \textsc{Link($H',z_i,z_j$)}, $f(z_i|\MM)$ and $f(z_j|\MM)$ each increases by at most 1.

Consider step $(s,i)$ of stage 1.  The vertex $z_s$ has been linked $i-p-1$ times and so from \eqref{eq:HH} we have $f(z_s|\MM)< p+s+i$, and $z_i$ has been linked $p-s-1$ times and so $f(z_i|\MM)< t+3p+r_p-s$. Then
$$f(z_s|\MM)+f(z_i|\MM)< t+4p+i+r_p\le 6t+r_p,$$
and so \eqref{equation:link} holds for every step $(s,i)$ of stage 1.
From  \eqref{eq:HH} and the definition of the algorithm, we have that at the end of stage 1:
\begin{equation}\notag
\begin{array}{ll}
f(z_s|\MM)\le 2p+s\ \ \ \ \ \ \ \ \ \ \ \ \ \ \ \ \ \ \ \ \ \ \ \text{ if } s\le q,\\
f(z_s|\MM)\le (2p+s) + q\ \ \ \ \ \ \ \ \ \ \ \ \ \ \ \ \text{ if } q<s\le p,\\
f(z_i|\MM)\le (t+2p+r_p)+(p-q)  \ \ \ \text{ if } i>p. \end{array}
\end{equation}

Consider step $(s, i)$ of stage 2.  
The vertex $z_s$ has been linked $2(i-p-1)$ times during stage 2 and so
$f(z_s|\MM)\le 2p+s+2q-2=2t+s-2$ (since $t=p+q$), and $z_i$ has been linked $2(q-s-1)$ times.   Thus
$f(z_i|\MM)\le 2t+2p+r_p-2s-2$. It follows that 
$$f(z_s|\MM)+f(z_i|\MM)\le 4t+2p+r_p-s-4\le 6t+r_p-4.$$
We can perform \textsc{Link$(H',z_s,z_i)$} twice. At the end of stage 2, we have 
\begin{equation}\notag
\begin{array}{ll}
f(z_s|\MM)\le (2p+s)+2q\le 2t+s\ \ \ \ \text{ if } s\le q,\\
f(z_s|\MM)\le (2p+s)+ q\le 2t+s\ \ \ \ \ \ \text{ if } q<s\le p. \end{array}
\end{equation}

Consider step $(s,i)$ of stage 3.
The vertex $z_s$ has been linked $(p-s)+(s-i-1)$ times during stage 3 (in which $p-s$ times with $z_{r},s<r\le p$ and $s-i-1$ with $z_{j},i<j\le s$) and so
$f(z_s|\MM)< (2t+s)+p-i$, and $z_i$ has been linked $p-s-1$ times and so
$f(z_i|\MM)<(2t+i)+p-s$. Then 
$$f(z_s|\MM)+f(z_i|\MM)< 4t+2p\le 6t+r_p,$$
and so \eqref{equation:link} holds for every step $(s,i)$ of stage 3.  Claim \ref{cl:HH} now follows.
\end{cproof}
This proves Lemma \ref{lemma:mainmindeg}, and so prove Theorem \ref{theorem:mindeg}.
\end{proof}

\section{Forcing a clique immersion via the chromatic number} \label{section:chromatic}

In this section we shall prove Theorem \ref{theorem:chromatic}. Recall that given $\ell \ge 1$, a graph $G$ is \emph{$\ell$-critical} if the chromatic number of $G$ is $\ell$, and deleting any vertex of $G$ results in a subgraph with chromatic number $\ell-1$. A well-known property of critical graphs is that if $G$ is a graph 
with chromatic number  $\ell$, then $G$ contains an $\ell$-critical subgraph. Let us restate Theorem \ref{theorem:chromatic}.

\begin{theorem}\label{theorem:chromatic2}
Every graph with chromatic number at least $3.54t+4$ contains an immersion of $K_t$.
\end{theorem}

\begin{proof}
Assume the theorem is false, and suppose that there exists a graph of chromatic number $\ell \ge 3.54t+4$ but which does not immerse $K_t$. 
Let $G^*$ be an $\ell$-critical subgraph of that graph. Let $v_0$ be a vertex of $G^*$ with minimum degree.  By Theorem \ref{theorem:mindeg}, $d_{G^*}(v_0)\le 7t+6$. 
Let $N=N_{G^*}(v_0)$, and let $G$ be the graph obtained from $G^*$ by deleting $v_0$.  It follows that $G$ does not immerse $K_t$.
The graph $G^*$ is $\ell$-critical, so $G$ has chromatic number $\ell-1$. Furthermore, for any coloring of $G$, $N$ always has at least one vertex of each of the $\ell-1$ colors, otherwise we could color $G^*$ with $\ell-1$ colors.  The proof uses Kempe-chains, introduced by Alfred Kempe in an 1879 attempt to prove that planar graphs are 4-colorable, to build a clique immersion with branch vertices in $N$.  

Given a coloring of $G$, we call a vertex $v\in N$ a \textit{singleton} if $v$ is the unique vertex in $N$ with its color. Two vertices $v,v'\in N$ of the same color form a \textit{doubleton} if they are the only two vertices with a given color in $N$. 
Let $\CC$ be an $\ell - 1$ coloring of $G$ which maximizes the number of singletons.
Let $X=\{x_1,...,x_\alpha\}$ and $Y=\{y_1,y_1',...,y_\beta,y'_\beta\}$ be the sets of singletons and doubletons, respectively, where $x_i$ has color $a_i$ and $y_i,y'_i$ share color $b_i$. 
All other colors appear at least 3 times in $N$. Thus, $\ell - 1$, the number of colors in $N$, is at most $$\alpha +\beta+ \frac{|N|-\alpha-2\beta}{3}=\frac{|N|+2\alpha+\beta}{3}.$$
Since $|N|= d_{G^*}(a)\le 7t+6$, we have
$$3.54t+3\le \ell-1 \le \frac{7t+6+2\alpha+\beta}{3}$$
\begin{equation}
\Longrightarrow\   2\alpha+\beta\ge 2.62t+3   \label{equation:2alpha}.
\end{equation}

Given colors $a,b$, an \textit{$(a,b)$-chain} is a path with vertices colored alternately by colors $a$ and $b$. Clearly, if $\{a,b\}\ne \{a',b'\}$, then any $(a,b)$-chain and $(a',b')$-chain are edge-disjoint. 
The idea is as follows. We first show that there are many chains with endpoints in $X\cup Y$. Since these chains are edge-disjoint, we can split them off to get a dense graph on $X\cup Y$, then apply Lemma \ref{lemma:average2} to obtain a $K_t$ immersion, which leads to the contradiction.

\begin{claim}\label{claim:chains}
The following hold.
\begin{enumerate}[label=(\alph*)]
\item For all pairs of distinct colors $a_i, a_j$, there is an $(a_i,a_j)$-chain from $x_i$ to $x_j$. \label{enumerate:sing}

\item For any colors $a_i,b_j$, there is an $(a_i,b_j)$-chain from $x_i$ to $y_j$, or from $x_i$ to $y'_j$. \label{enumerate:sing-doub}

\item \label{enumerate:doub} For all pairs of distinct colors $b_i,b_j$, one of the following holds:
\begin{enumerate}[label=(\roman*)]
\setcounter{enumi}{2}
\item \label{enumerate:alph1}  there exist two edge-disjoint $(b_i,b_j)$-chains linking $y_i$ to $y_j$ and $y'_i$ to $y'_j$;

\item  \label{enumerate:alph2} there exist two edge-disjoint $(b_i,b_j)$-chains linking $y_i$ to $y'_j$ and $y'_i$ to $y_j$;

\item \label{enumerate:alph3} there exist $(b_i,b_j)$-chains from any of $y_i,y_i'$ to any of $y_j,y_j'$ but they cannot be chosen edge-disjoint.
\end{enumerate}
\end{enumerate}
\end{claim}

\begin{cproof}
For every color $a$, let $V_{a}\subseteq V(G)$ be the set of all vertices of color $a$ in $\CC$.

To prove \ref{enumerate:sing}, suppose that there exist two distinct colors $a_i,a_j$ such that there is no $(a_i,a_j)$-chain from $x_i$ to $x_j$. Then $x_i,x_j$ are disconnected in $G[V_{a_i}\cup V_{a_j}]$.
Let $U$ be the connected component containing $x_i$ in $G[V_{a_i}\cup V_{a_j}]$. We exchange the color of all vertices in $U$ from color $a_i$ to $a_j$ and vice versa and obtain a new coloring $\CC'$ in $G$. Clearly $\CC'$ is a proper coloring in  $G[V_{a_i}\cup V_{a_j}]$, and so is a proper coloring in $G$.
Now both $x_i$ and $x_j$ has color $a_j$, so $\CC'$ has no vertex of color $a_i$, contrary to the fact that $N$ has all colors for every $(\ell-1)$-coloring of $G$. 

To prove \ref{enumerate:sing-doub}, the same argument works. Suppose that there exist two distinct colors $a_i,b_j$ such that there is no $(a_i,b_j)$-chain from $x_i$ to $\{y_j,y_j'\}$. Then $x_i$ is disconnected with $\{y_j,y_j'\}$ in $G[V_{a_i}\cup V_{b_j}]$.
Let $U$ be the connected component containing $x_i$ in $G[V_{a_i}\cup V_{b_j}]$. We exchange the color of all vertices in $U$ from color $a_i$ to $b_j$ and vice versa and obtain a new coloring $\CC'$ in $G$. Then $\CC'$  is a proper coloring in $G$ and has smaller number of colors on $N$ than $\CC$, contrary to the fact that $N$ has all colors for every $(\ell-1)$-coloring of $G$. 

To prove \ref{enumerate:doub}, we first prove that
\begin{enumerate}[label=(\alph*)]\setcounter{enumi}{3}
\item \label{en:d} for every pair of distinct colors $b_i,b_j$, there is a $(b_i,b_j)$-chain from $y_i$ to $y_j$, or from $y_i$ to $y'_j$.
\end{enumerate}

Suppose that there exist two distinct colors $b_i,b_j$ such that there is no $(b_i,b_j)$-chain from $y_i$ to $\{y_j,y_j'\}$. Then $y_i$ is disconnected with $\{y_j,y_j'\}$ in $G[V_{b_i}\cup V_{b_j}]$.
Let $U$ be the connected component containing $y_i$ in $G[V_{y_i}\cup V_{b_j}]$. We exchange the color of all vertices in $U$ from color $b_i$ to $b_j$ and vice versa and obtain a new coloring $\CC'$ in $G$. Then $\CC'$  is a proper coloring in $G$.
If $y_i'\in U$, then $\CC'$  has no vertex of color $b_i$ in $N$, contrary to the fact that $N$ has all colors for every $(\ell-1)$-coloring of $G$. 
If $y_i'\notin U$, then $\CC'$ has exactly one vertex of color $b_i$ in $N$, and so has more singletons than $\CC'$, which contradicts our choice of $\CC$ to maximize the number of singletons.

We now show how \ref{en:d} implies \ref{enumerate:doub}. From \ref{en:d}, every pair of distinct colors $b_i,b_j$, there is a $(b_i,b_j)$-chain from $y_i$ to $\{y_j,y_j'\}$ and another $(b_i,b_j)$-chain from $y_i'$ to $\{y_j,y_j'\}$. 
If one chain go to $y_j$ and another go to $y_j'$, then there are three possibilities.
First, these chains are edge-disjoint and between $y_i,y_j$ and $y_i',y_j'$, then \ref{enumerate:alph1} holds. Second,  these chains are edge-disjoint and between $y_i,y_j'$ and $y_i',y_j$, then \ref{enumerate:alph2} holds. Third, they are not edge-disjoint, then all $\{y_i,y_i',y_j,y_j'\}$ are connected by these two chains, and \ref{enumerate:alph3} holds.
Otherwise, say these chains both go from $y_i,y_i'$ to $y_j$. Then by \ref{en:d}, there is a $(b_i,b_j)$-chain from $y_j'$ to either $y_i$ or $y_i'$. Hence all $\{y_i,y_i',y_j,y_j'\}$ are connected by some $(b_i,b_j)$-chains, and \ref{enumerate:alph3} holds.
\end{cproof}

For every pair of colors, we fix a subgraph based on the appropriate outcome of Claim \ref{claim:chains}.  For every $i, j$, $1 \le i < j \le \alpha$, fix $C_a(i,j)$ to be an $(a_i, a_j)$-chain from $x_i$ to $x_j$.  For all $i, j$, $1 \le i \le \alpha$, $1 \le j \le \beta$, fix $C_b(i,j)$ to be an $(a_i, b_j)$-chain from $x_i$ to either $y_j$ or $y_j'$.  Let $i, j$ be such that $1 \le i < j \le \beta$; one of \ref{enumerate:alph1} - \ref{enumerate:alph3} holds for the colors $b_i$ and $b_j$.  If either \ref{enumerate:alph1} or \ref{enumerate:alph2} holds, fix $C_c(i, j)$ to be the subgraph consisting of two edge disjoint $(b_i, b_j)$-chains linking $\{y_i, y_i'\}$ and $\{y_j, y_j'\}$.  If \ref{enumerate:alph3} holds, fix $C(i,j)$ to be an edge minimal subgraph containing $(b_i, b_j)$-chains linking each of $y_i, y_i'$ to each of $y_j, y_j'$.
For $i, j$, $1 \le i < j \le \beta$, we say that $C_c(i,j)$ has one of 3 \emph{types}, namely \ref{enumerate:alph1}, \ref{enumerate:alph2}, or \ref{enumerate:alph3}, depending on which outcome of \ref{enumerate:doub} holds.  Note that $C_a(i,j)$, $C_b(i,j)$, and $C_c(i,j)$ are all pairwise edge disjoint.  

If we split off all the possible edge disjoint paths contained in subgraphs from the previous paragraph, it will not necessarily be the case that we will have sufficient edges on $X\cup Y$ to apply Lemma \ref{lemma:average2}.  To get around this problem, we focus instead on the vertex set $X \cup \{y_1, \dots, y_\beta\}$.    The subgraphs $C_c(i,j)$ of type \ref{enumerate:alph1} or type \ref{enumerate:alph3} contain a path which can be split off to yield the edge $y_iy_j$.  Moreover, if we flip the labels $y_i$ and $y_i'$, every $C_c(i,j)$ subgraph of type \ref{enumerate:alph2} becomes a $C_c(i,j)$ subgraph of type \ref{enumerate:alph1} (and vice versa).  Thus, we can increase the density of the resulting graph on $X \cup \{y_1, \dots, y_\beta\}$ by flipping the appropriate pairs of labeles $y_i$, $y_i'$.  

Unfortunately, this greedy approach will still not yield enough edges on $X \cup \{y_1, \dots, y_\beta\}$ to apply Lemma \ref{lemma:average2}.  To further increase the final edge density, we will group together multiple $C_c(i,j)$ subgraphs of type \ref{enumerate:alph2} to split off paths and add further edges to the set $\{y_1, \dots, y_\beta\}$.  The remainder of the argument carefully orders how the subgraphs are grouped together so that when we split them off and get as dense a subgraph as possible on the vertex set $X \cup \{y_1, \dots, y_\beta\}$.

We begin by defining the subgraphs $G_1$, $G_2$ and the auxiliary graph $H$ as follows.  Split off all paths of the form $C_a(i, j)$, $C_b(i, j)$, and the two edge disjoint $\{y_iy_i'\}-\{y_j,y_j'\}$-paths contained in the subgraphs $C_c(i,j)$ of type \ref{enumerate:alph1} and \ref{enumerate:alph2}.   Let $G_1$ the graph with vertex set $V(G)$ and edge set the set of all new edges arising from splitting off these paths. Let $G_2$ be the subgraph of $G$ with vertex set $V(G)$ edge set the union of $E(C_c(i,j))$ for all subgraphs $C_c(i,j)$ of type \ref{enumerate:alph3}. Observe that $G_1\cup G_2$ is an immersion of $G$ and therefore does not immerse $K_t$. 
Clearly, $G_1[X]$ is a complete graph obtained from splitting off all the subgraphs $C_a(i,j)$, and so
\begin{equation}\label{claim:alpha}
\alpha=|X|\le t-1.
\end{equation}

We define an auxiliary graph $H$ by replacing each pair of vertices $y_i,y_i'$ with a single vertex $z_i$, and we color edges of incident with $z_i$ to describe the behavior of $y_i,y_i'$. Precisely, let $H$ be a graph
with vertex set $X \cup Z$ where $Z=\{z_1,...,z_\beta\}$ and edge set 
\begin{align*}
E(H) = \{x_iz_j: 1\le i&\ \le \alpha,  1 \le j \le \beta\}\cup \\
\cup \{z_iz_j:&\ 1 \le i < j \le \beta \text{ and $C_c(i,j)$ is not of type \ref{enumerate:alph3}}\}.
\end{align*}
The edges of $H$ are improperly colored by two colors \textit{odd, even} as follows:
\begin{itemize}
\item $x_iz_j$ is even if $x_iy_j\in E(G_1)$, and is odd if $x_iy'_j\in E(G_1)$.
\item $z_iz_j$ is even if $y_iy_j,y_i'y_j'\in E(G_1)$, and is odd if $y_iy_j',y_i'y_j\in E(G_1)$.
\end{itemize}

To perform a \emph{swap} at a vertex $z_i$, we exchange the colors of all edges incident with $z_i$ in $H$; a swap is equivalent to switching the labels of $y_i$ and $y_i'$ in $G_1\cup G_2$. 
To \emph{swap} a set $S\subseteq Z$, we swap vertices in $S$ sequentially in an arbitrarily chosen order.  One can easily show that to swap a set $S$ is equivalent to switching the color of every edges between $S$ and $V(H)\backslash S$.

A triangle in $H$ is \textit{odd} if it has odd number of odd-edges.  A key property of odd-triangles is that an odd-triangle is still odd after any swap. 
Each odd-triangle either has 3 vertices in $Z$ or exactly two vertices in $Z$ -- call them type 1 and type 2 odd-triangles, respectively. Given a type 1 odd-triangle $z_iz_jz_k$, the set of edges in $G_1$ with endpoints in $\{y_i,y_i',y_j,y'_j,y_k,y'_k\}$ are called the \emph{corresponding edges} of $z_iz_jz_k$. Similarly, given a type 2 odd-triangle $x_iz_jz_k$, the set of edges in $G_1$ with endpoints in $\{x_i,y_j,y_j',y_k,y'_k\}$ are called the \emph{corresponding edges} of $x_iz_jz_k$. Clearly, the set of corresponding edges of two edge-disjoint odd-triangles are disjoint. In Figure \ref{fig}, we describe all possibilities (up to permutation of indices) of the set of corresponding edges of a type 1 odd-triangle (upper figures) and of a type 2 odd-triangle (lower figures). 

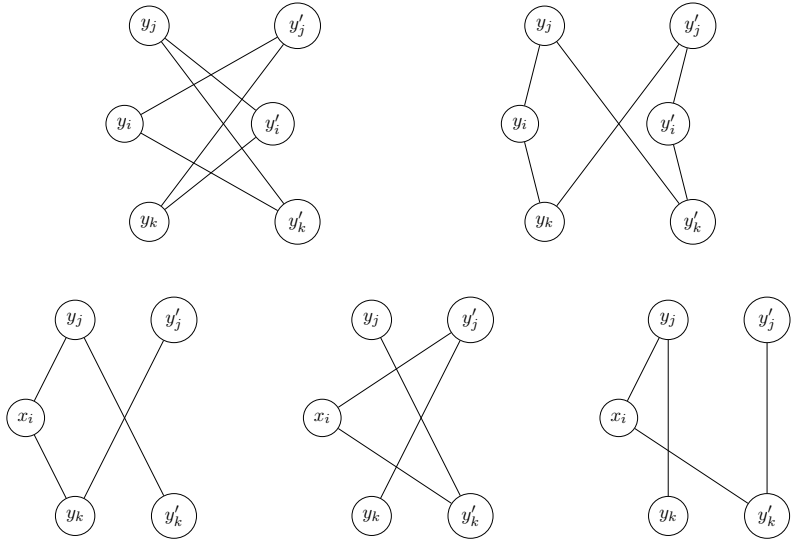
\begin{figure}[h]
\begin{center}
\scalebox{0.65}{
\begin{tikzpicture}
    \node[shape=circle,draw=black] (A) at (0,0) {$y_k$};
    \node[shape=circle,draw=black] (B) at (0,4) {$y_j$};
    \node[shape=circle,draw=black] (C) at (-0.5,2) {$y_i$};
    \node[shape=circle,draw=black] (D) at (3,0) {$y_k'$};
    \node[shape=circle,draw=black] (E) at (3,4) {$y_j'$};
    \node[shape=circle,draw=black] (F) at (2.5,2) {$y_i'$} ;

    \path [-] (A) edge node[left] {} (E);
    \path [-](A) edge node[left] {} (F);
    \path [-](B) edge node[left] {} (D);
    \path [-](B) edge node[left] {} (F);
    \path [-](C) edge node[right] {} (E);
    \path [-](C) edge node[left] {} (D);  

\begin{scope}[shift={(8cm,0cm)}], 
    \node[shape=circle,draw=black] (A) at (0,0) {$y_k$};
    \node[shape=circle,draw=black] (B) at (0,4) {$y_j$};
    \node[shape=circle,draw=black] (C) at (-0.5,2) {$y_i$};
    \node[shape=circle,draw=black] (D) at (3,0) {$y_k'$};
    \node[shape=circle,draw=black] (E) at (3,4) {$y_j'$};
    \node[shape=circle,draw=black] (F) at (2.5,2) {$y_i'$} ;

    \path [-] (A) edge node[left] {} (E);
    \path [-](E) edge node[left] {} (F);
    \path [-](B) edge node[left] {} (D);
    \path [-](D) edge node[left] {} (F);
    \path [-](C) edge node[right] {} (A);
    \path [-](C) edge node[left] {} (B);
\end{scope} 

\begin{scope}[shift={(-1.5cm,-6cm)}],
 
    \node[shape=circle,draw=black] (A) at (0,0) {$y_k$};
    \node[shape=circle,draw=black] (B) at (0,4) {$y_j$};
    \node[shape=circle,draw=black] (C) at (-1,2) {$x_i$};
    \node[shape=circle,draw=black] (D) at (2,0) {$y_k'$};
    \node[shape=circle,draw=black] (E) at (2,4) {$y_j'$};
    \path [-] (A) edge node[left] {} (E);
    \path [-](B) edge node[left] {} (D);
    \path [-](C) edge node[right] {} (B);
    \path [-](C) edge node[left] {} (A); 

\begin{scope}[shift={(6cm,0cm)}], 
    \node[shape=circle,draw=black] (A) at (0,0) {$y_k$};
    \node[shape=circle,draw=black] (B) at (0,4) {$y_j$};
    \node[shape=circle,draw=black] (C) at (-1,2) {$x_i$};
    \node[shape=circle,draw=black] (D) at (2,0) {$y_k'$};
    \node[shape=circle,draw=black] (E) at (2,4) {$y_j'$};
    \path [-] (A) edge node[left] {} (E);
    \path [-](B) edge node[left] {} (D);
    \path [-](C) edge node[right] {} (D);
    \path [-](C) edge node[left] {} (E); 
\end{scope}

\begin{scope}[shift={(12cm,0cm)}], 
    \node[shape=circle,draw=black] (A) at (0,0) {$y_k$};
    \node[shape=circle,draw=black] (B) at (0,4) {$y_j$};
    \node[shape=circle,draw=black] (C) at (-1,2) {$x_i$};
    \node[shape=circle,draw=black] (D) at (2,0) {$y_k'$};
    \node[shape=circle,draw=black] (E) at (2,4) {$y_j'$};
    \path [-] (A) edge node[left] {} (B);
    \path [-](E) edge node[left] {} (D);
    \path [-](C) edge node[right] {} (B);
    \path [-](C) edge node[left] {} (D); 
\end{scope}
\end{scope}

\end{tikzpicture}
}
\end{center}
\caption{Possibilities of corresponding edges of odd-triangles.}\label{fig}
\end{figure}

Looking at Figure \ref{fig}, we can easily verify the following.
\begin{enumerate}[label=(\Alph*)]
\item \label{enumerate:Alph1} If $z_iz_jz_k$ is an odd-triangle of type 1, we can split off its corresponding edges to obtain edges $y_iy_j,y_jy_k,y_ky_i$. 

\item \label{enumerate:Alph3} If $x_iz_jz_k$ is an odd-triangle of type 2, we can split off its corresponding edges to obtain the edge $y_jy_k$.
\item \label{enumerate:Alph2} 
If $x_iz_jz_k$ is an odd-triangle of type 2, we can alternatively split off its corresponding edges to obtain two edges from the set$\{x_iy_j,y_jy_k,y_kx_i\}$ (exactly which two edges depends on which case from Figure \ref{fig} we find ourselves in).
\end{enumerate}

Let $H_1$ be a graph obtained from $H$ by removing an (inclusion-wise) maximal set $\mathcal{T}_1$ of pairwise edge-disjoint odd-triangles of type 1, and 
let $H_2$ be a graph obtained from $H_1$ by removing an (inclusion-wise) maximal set $\mathcal{T}_2$ of pairwise edge-disjoint odd-triangles of type 2. In the following claims, we employ the assumption that $G$ does not contain a $K_t$-immersion to bound the degree of vertices in $H_1[Z]$ and $H_2$.

\begin{claim}\label{claim:odd1}
$d_{H_1[Z]}(z)< t$ for every $z\in Z$.
\end{claim}
\begin{cproof}
Suppose for a contradiction that there exists $z\in Z$ such that $d_{H_1[Z]}(z)\ge t$. 
Let $M_o$ ($M_e$) the sets of vertices adjacent to $z$ in $H_1[Z]$ by an odd-edge (by an even-edge, respectively). Then $|M_o|+|M_e|=d_{H_1[Z]}(z)\ge t$.   
Every edge $uv$ in $H_1[Z]$ with $u,v\in M_o$ ($u,v\in M_e$, respectively) must be even; otherwise, $uvz$ is an odd-triangle of type 1, contradicting the maximality assumption on ${\cal T} _1$.
Similarly, every edge $uv$ in $H_1[Z]$ with $u\in M_o$ and $v\in M_e$ must be odd. 

We now swap $M_o$, and then the new graph $H_1[M_o\cup M_e]$ contains only even-edges.
Let $M=\{y_i:z_i\in M_o\cup M_e\}$. Then $|M|=|M_o|+|M_e|\ge t$.
For every odd-triangle in $\mathcal{T}_1$, we split off corresponding edges in $G_1$ by method \ref{enumerate:Alph1} to get $y_iy_j,y_jy_k,y_ky_i$. Then for any distinct vertices $y_i,y_j\in M$, we have
\begin{itemize}
\item if $z_iz_j\in H_1$, then $z_iz_j$ is even, and hence $y_iy_j\in G_1$. 
\item if $z_iz_j\in H\backslash H_1$, then $z_iz_j$ belongs to some odd-triangle in $\mathcal{T}_1$, and we showed above that we can obtain $y_iy_j$ by splitting off edges of $G_1$ by method \ref{enumerate:Alph1}.
\item if $z_iz_j\notin H$, then $C_c(i,j)$ is of type \ref{enumerate:alph3} and so there exists a $y_i - y_j$ path in $C_c(i,j)$ which can be split off to yield the edge $y_iy_j$. 
\end{itemize} 
We end up with a complete graph on $M$, and so conclude that $G_1\cup G_2$ contains $K_t$ as an immersion (since $|M|\ge t$), which is a contradiction.
\end{cproof}

\begin{claim}\label{claim:odd2}
$d_{H_2}(x)<t$ for every $x\in X$. 
\end{claim}

\begin{cproof}
The proof is quite similar to the proof of Claim \ref{claim:odd1}.  We suppose that there exists $x\in X$ such that $d_{H_2}(x)\ge t$. Note that $X$ is a stable set in $H$, and so all neighbors of $x$ in $H_2$ are in $Z$.
Let $M_o$ ($M_e$) be the set of vertices adjacent to $x$ in $H_2$ by an odd-edge (by an even-edge, respectively). Then $M_o\cup M_e\subseteq Z$ and $|M_o|+|M_e|\ge t$.   
Every edge $uv$ in $H_2$ with $u,v\in M_o$ ($u,v\in M_e$, respectively) must be even; otherwise, $uvx$ is an odd-triangle of type 2, contradicting the maximality assumption of ${\cal T} _2$.
Similarly, every edge $uv$ in $H_2$ with $u\in M_o$ and $v\in M_e$ must be odd. 

We now swap $M_o$.  The new graph $H_2[M_o\cup M_e]$ contains only even-edges.
Let $M=\{y_i:z_i\in M_o\cup M_e\}$. Then $|M|=|M_o|+|M_e|\ge t$.
For every odd-triangle in $\mathcal{T}_1$, we split off corresponding edges in $G_1$ by method \ref{enumerate:Alph1} to get $y_iy_j,y_jy_k,y_ky_i$.
For every odd-triangle in $\mathcal{T}_2$, we split off corresponding edges in $G_1$ by method \ref{enumerate:Alph3} to get $y_iy_j$.
Then for any distinct vertices $y_i,y_j\in M$, we have
\begin{itemize}
\item if $z_iz_j\in H_2$, then $z_iz_j$ is even, and hence $y_iy_j\in G_1$. 
\item if $z_iz_j\in H_1 - E(H_2)$, then $z_iz_j$ belongs to some odd-triangle in $\mathcal{T}_2$, and we showed above that we can obtain $y_iy_j$ by splitting off edges of $G_1$ by method \ref{enumerate:Alph3}.
\item if $z_iz_j\in H - E(H_1)$, then $z_iz_j$ belongs to some odd-triangle in $\mathcal{T}_1$, and we showed above that we can obtain $y_iy_j$ by splitting off edges of $G_1$ by method \ref{enumerate:Alph1}.
\item if $z_iz_j\notin H$, we split off a $y_i - y_j$ path in $C_c(i,j)$ in $G_2$ to obtain $y_iy_j$.
\end{itemize} 
We end up with a complete on $M$, and so $G_1\cup G_2$ contains $K_t$ as an immersion (since $|M|\ge t$), which is a contradiction.
\end{cproof}

The next claim guarantees that at least half of edges in $H_2$ are even. 

\begin{claim}\label{claim:bigswap}
There exists a subset $S$ of vertices such that after swapping $S$ in $H_2$, the number of even-edges in $H_2$ is at least the number of odd-edges.
\end{claim}
\begin{cproof}
We first show that there exits a sequence of swaps resulting in the number of even-edges in $H_2[Z]$ being at least the number of odd-edges $H_2[Z]$.
If there is $z\in Z$ such that $z$ is incident with more odd-edges than even-edges in $H_2[Z]$, we swap $z$, then repeat. 
The process will halt since the number of even-edges in $H_2[Z]$ strictly increases after each swap. 
When the process halts, every $z\in Z$ is incident with at least as many even-edges as with odd-edges in $H_2[Z]$, and so in total, the number of even-edges in $H_2[Z]$ at least the number of odd-edges $H_2[Z]$.

If the number of even-edges from $Z$ to $X$ in $H_2$ is less than the number of odd-edges from $Z$ to $X$ in $H_2$, we swap the set $Z$.  After the switch,  the number of even-edges from $Z$ to $X$ in $H_2$ is at least the number of odd-edges from $Z$ to $X$ in $H_2$.  Moreover, edges in $H_2[Z]$ are not affected by swapping $Z$.  Finally, note that there is no edge in $H_2[X]$.

Thus at the end of this series of swaps, the number of even-edges in $H_2$ is at least the number of odd-edges in $H_2$, proving the claim.
\end{cproof}

By Claim \ref{claim:bigswap}, we may assume that at least half of edges in $H_2$ are even. 
We now split off edges in $G_1\cup G_2$ to obtain a dense graph on $X\cup \{y_1,...,y_\beta\}$ as follows. 
For every odd-triangle in $\mathcal{T}_1$, we split off its corresponding edges in $G_1$ by method \ref{enumerate:Alph1}. For every odd-triangle in $\mathcal{T}_2$, we split off its corresponding edges in $G_1$ by method \ref{enumerate:Alph2}, which implies that we obtain two of three edges in the set $\{x_iy_j, y_jy_k, y_kx_i\}$. 
For every pairs $b_i,b_j$ in case \ref{enumerate:alph3}, we also split off a path in $C_c(i,j)$  to get the edge $y_iy_j$ as guaranteed by \ref{enumerate:alph3}.
We denote by $\hat G$ the induced subgraph of the new graph on $X\cup \{y_1,...,y_\beta\}$. Note that $\hat G$ is an immersion of $G_1\cup G_2$, and so does not contain an immersion of $K_t$.

We will show that $\hat G$ is dense, specifically by counting the number of non-edges in $\hat G$.  We first observe that by construction, $\hat G[X]$ is complete.  
Thus, all non-edges in $\hat G$ arise from odd-edges of $H$ for which the corresponding edge of $\hat G$ cannot be reconstructed through odd-triangles.

Observe that $H=(H - E(H_1)) \cup (H_1 - E(H_2)) \cup H_2$.  We consider each of the subgraphs $H - E(H_1)$,  $H_1 - E(H_2)$, and $H_2$ and how they can contribute non-edges to $\hat G$ separately.
\begin{itemize}
\item Each odd-triangle $z_iz_jz_k\in \mathcal{T}_1$ contributes zero missing edge to $\hat G$ since we obtain $y_iy_j,y_jy_k,y_ky_i$ by method \ref{enumerate:Alph1}. Hence $H - E(H_1)$ (the union of odd-triangles in $\mathcal{T}_1$) contributes zero missing edge to  $\hat G$.

\item Each odd-triangle $x_iz_jz_k\in \mathcal{T}_2$ contributes exactly one missing edge to $\hat G$ since we obtain two edges among $x_iy_j,y_jy_k,y_kx_i$ by method \ref{enumerate:Alph2}. Hence $H_1 - E( H_2)$ (the union of odd-triangles in $\mathcal{T}_2$) contributes $|\mathcal{T}_2|$ missing edge to  $\hat G$.

\item Each odd-edge (even-edge) in $H_2$ contributes exactly one (zero, respectively) missing edge to $\hat G$. Hence $H_2$ contributes at most $|E(H_2)|/2$ missing edges to  $\hat G$ by Claim \ref{claim:bigswap}.
\end{itemize} 
We conclude that the number of missing edges in $\hat G$ is at most $|E(H_2)|/2+ |\mathcal{T}_2|$. We next give an explicit bound for the number of missing edges in $\hat G$.

\begin{claim}\label{cl:df}
The number of missing edges in $\hat G$ is at most $(\alpha\beta+\alpha t+\beta t)/4$.
\end{claim}
\begin{cproof}

Let $p= |E(H_2)|/2$ and $q=|\mathcal{T}_2|$. Then the number of missing edges in $\hat G$ is at most $p+q$.
By claim \ref{claim:odd1}, $H_1[Z]$ has $\beta$ vertices and minimum degree less than $t$, and so $E(H_1[Z])< \beta t/2$. Hence
\begin{align*}
2p+3q &\le |E(H_2)|+|E(H_1 - E(H_2))|\\
&= |E(H_1)|\\
&= |X||Z|+\big|E(H_1[Z])\big|\\
&\le \alpha\beta+\beta t/2. 
\end{align*}

Claim \ref{claim:odd2} states that every $x\in X$ is adjacent to at most $t$ vertices of $Z$ in $H_2$, and so is adjacent to at most $|Z|-t$ vertices of $Z$ in $H_1 - E(
H_2)$.
This implies that for every $x\in X$, there are at least $(|Z|-t)/2=(\beta-t)/2$ odd-triangles in $\mathcal{T}_2$ containing $x$ (since $H_1 - E(H_2)$ is the union of odd-triangles in $\mathcal{T}_2$). This means that $$q=|\mathcal{T}_2|\ge |X|(\beta-t)/2= \alpha(\beta-t)/2.$$
Hence 
$$p+q = \frac{(2p+3q)-q}{2}\le \frac{(\alpha\beta+\beta t/2)-\alpha(\beta-t)/2}{2}=\frac{\alpha\beta+\alpha t+\beta t}{4}.$$
Hence the number of missing edges in $\hat G$ is at most $(\alpha\beta+\alpha t+\beta t)/4$.
\end{cproof}

We next show that if $|V(\hat G)|=\alpha+\beta$ is large, then we can apply Lemma \ref{lemma:average2} to yield a contradiction that $\hat G$ contains an immersion of $K_t$.  Hence $\alpha+\beta$ is small, which contradicts \eqref{equation:2alpha}, and the proof of Theorem \ref{theorem:chromatic} is complete.

\begin{claim}\label{claim:alphabeta}
$\alpha+\beta<2.62(t+1)$.
\end{claim}
\begin{cproof}
Let $n=|V(\hat G)|=\alpha+\beta$ and suppose for a contradiction that $n\ge2.62(t+1)$. Let $\ff=\frac{1}{n}\sum_{v\in \hat G}f_{\hat G}(v)$. Then $n\ff/2$ is the number of missing edges in $\hat G$, and so by Claim \ref{cl:df} we have
\begin{equation}
2\ff\le\frac{\alpha\beta+\alpha t+\beta t}{n}=\frac{\alpha\beta}{n} +t 
\label{equation:ff}.
\end{equation}

If $\ff<n/4$, then by Lemma \ref{lemma:average2}, $\hat G$ contains an immersion of $K_{t'}$, where $t'=\lfloor n/2\rfloor\ge t$.
Hence $\hat G$ contains an immersion of $K_t$, a contradiction. 

Otherwise, since $\alpha\beta \le (\alpha+\beta)^2/4=n^2/4$, we have $2\ff<n/4 +t<n$. Thus by applying Lemma \ref{lemma:average2}, $\hat G$ contains an immersion of $K_{t'}$, where $t'=\lfloor n-2\ff \rfloor> n-2\ff-1$. 
We conclude that $n-2\ff-1<t$ since $\hat G$ does not contain $K_t$ as an immersion.

Recall that by \eqref{claim:alpha}, we have that $\alpha<n/2$. For every $x$ such that $\alpha<x<n/2$, we have 
$$\alpha\beta =\alpha(n-\alpha) <x(n-x).$$
Since $\alpha<t+1<n/2$, we can choose $x:=t+1$, and so
\begin{align*}
(n-2\ff-1)-t &\ge n-\bigg(\frac{\alpha\beta}{n}+t\bigg)-t-1\\
&\ge n-\frac{\alpha\beta}{n}-2x\\
&> \frac{n^2-3nx+x^2}{n}.
\end{align*}
The assumption of the claim is that $n\ge 2.62x$, and hence $n^2-3nx+x^2\ge 0$ (by solving the quadratic equation). This gives $n-2\ff-1\ge t$, which contradicts what we obtained above that $n-2\ff-1<t$. This prove the claim.
\end{cproof}
Combining Claim \ref{claim:alphabeta} with \eqref{claim:alpha}, we obtain $2\alpha+\beta< 3.62t+3$, which contradicts \eqref{equation:2alpha}. This completes the proof of Theorem \ref{theorem:chromatic}.
\end{proof}


\section{Immersion in graphs with no stable set of size 3} \label{section:2.5}

We begin by reformulating Theorem \ref{theorem:2.5}.

\begin{theorem}\label{theorem:2.5.2}
For all $t \ge 1$, every graph $G$ with at least $5t$ vertices and no stable set of size three has a strong immersion of $K_{2t}$.
\end{theorem}
\begin{proof}
Assume that the theorem is false, and pick a counterexample $G$ which minimizes $|V(G)| + |E(G)|$.  Assume that $G$ has at least $5t+5$ vertices and no strong immersion of $K_{2t+2}$.  Since every graph on at least 5 vertices with no independent set of size three contains an edge, we may assume that $t \ge 1$.  By minimality, we may assume that $n = |V(G)| = 5t+5$.  Furthermore, as $G-e$ does not contain a strong immersion of $K_{2t+2}$ for all edges $e$, by minimality it follows that deleting any edge results in a stable set of size three.  All index arithmetic in the following proof is done mod 5.

\begin{claim}
$G$ contains an induced cycle of length $5$.
\end{claim}
\begin{cproof}
If $G$ were the disjoint union of cliques, since it contains no stable set of size three then it must be a disjoint union of at most two cliques. One of the two cliques has at least $\lceil n/2\rceil \ge 2t+2$ vertices, and so $G$ contains a strong $K_{2t+2}$-immersion, a contradiction.

Thus $G$ is not a disjoint union of cliques, and there exist two adjacent vertices $a_1,a_2$ such that $N(a_1)\ne N(a_2)$. Without loss of generality, we may suppose that $N(a_2)\backslash N(a_1)\ne \emptyset$ and let $a_3\in N_G(a_2)\backslash N_G(a_1)$. 
This gives $a_1a_3\notin E$ and $a_2a_3\in E$. 
Observe that there is $a_4$ with $a_1a_4,a_2a_4\notin E$; otherwise, we can remove the edge $a_1a_2$ without creating any stable set of size three, which contradicts the minimality of $G$.
By the same argument, there is $a_5$ with $a_2a_5,a_3a_5\notin E$. Note that $G$ does not contain any stable set of size three and $a_1a_4,a_1a_3\notin E$, and so $a_3a_4\in E$. Similarly, $a_1a_5,a_4a_5\in E$. Thus $a_1a_2a_3a_4a_5$ forms an induced cycle of length 5 in $G$.
\end{cproof}

Let $C=\{a_i:1\le i\le 5\}$ induce a cycle of length five, and let $U=V\backslash C$. Then $G[U]$ has $n-5 =  5t$ vertices and $G[U]$ contains no stable set of size three. By minimality of $G$, $G[U]$ contains a strong immersion of $K_{2t}$ with with some set of branch vertices $M$. 
Let $Q=U\backslash M$, and for every $i, 1\le i\le 5$, let $M_i$ be the set of vertices in $M$ not adjacent to $a_i$.

In the following claim, we show that if there are two large disjoint sets $X_1,X_3$ in $Q$ with some desired property, then for every $v\in M$, we can split off paths $a_1xv$ or $a_1xa_iv$ (of length 2 or 3) with $x\in X_1,i\in\{2,4,5\}$ to get the edge $a_1v$, and similarly to get the edge $a_3v$, and so get a strong clique immersion of size $2t+2$ on $M\cup \{a_1,a_3\}$, which is a contradiction.

\begin{claim}\label{claim:2.5X}
Suppose that there are disjoint sets $X_1,X_3\subseteq Q$ satisfying
\begin{enumerate}[label=(\roman*)]

\item \label{enum:2.5.2} $|X_1|\ge |M_1|$, and $|X_3|\ge |M_3|$; 
\item \label{enum:2.5.1} 
for every $x\in X_1$, we have $xa_1,xa_5\in E$ and either $xa_2\in E$ or $xa_4\in E$; and
\item \label{enum:2.5.3} for every $x\in X_3$, we have $xa_3,xa_4\in E$ and either $xa_2\in E$ or $xa_5\in E$.

\end{enumerate}
Then $G$ has a strong immersion of $K_{2t}$, where the set of branch vertices is $M\cup \{a_1,a_3\}$.
\end{claim}
\begin{cproof}
Let $E_1$ be the set of edges in $G$ from $C$ to $M_1\cup X_1$.
We wish to split off paths in $E_1$ to obtain edges from $a_1$ to every vertex in $M_1$. The process of splitting off is as follows.

\begin{itemize}
\item Arbitrarily pair each vertex $v\in M_1$ with a vertex $x_v\in X_1$ such that $x_v\ne x_{v'}$ for every $v\ne v'$ (such a choice of $x_v$ exists by \ref{enum:2.5.2}).

\item For every $v\in M_1$, note that $va_4\in E$ (otherwise, $\{a_1,v,a_4\}$ is a stable set of size three) and either $va_2\in E$ or $va_5\in E$ (otherwise, $\{a_2,v,a_5\}$ is a stable set of size three).
If $va_5\in E$, we split off the path $va_5x_va_1$ to get an edge $va_1$.
\item Otherwise, $va_4\in E$ and $va_2\in E$. By \ref{enum:2.5.1}, either $x_va_2\in E$ or $x_va_4\in E$.
If $x_va_2\in E$, we split off the path $va_2x_va_1$ to get the edge $va_1$. Otherwise, we split off the path $va_4x_va_1$ to get the edge $va_1$.
\end{itemize}
Note that in this process we only use edges of $E_1$ and at the end we obtain all edges from $a_1$ to $M_1$, and so obtain all edges from $a_1$ to $M$.
Let $E_3$ be the set of edges in $G$ from $C$ to $M_1\cup X_3$. Note that $E_1\cap E_3=\emptyset$, and hence we can split off paths in $E_3$ in the same manner to obtain all edges from $a_3$ to $M_3$, and so obtain all edges from $a_3$ to $M$.

By minimality, we can split off edges of $G[U]$ to obtain a $K_{2t}$ on $M$. Note that $E_1$, $E_3$ and $E(G[U])$ are pairwise disjoint, so we never split off an edge twice. By splitting off $a_1a_2,a_2a_3$, we obtain $a_1a_3$, and hence obtain a complete graph on $M\cup\{a_1,a_3\}$. 
Clearly, all split off paths are internally edge-disjoint from $M\cup\{a_1,a_3\}$. 
Hence $G$ contains a strong immersion of $K_{2t+2}$, where the set of branch vertices is $M\cup\{a_1,a_3\}$, a contradiction.
\end{cproof}

To reach the contradiction, it only remains to show that such sets $X_1,X_3$ exist up to shifting indices. 
For every $i,1\le i\le 5$, let $A_i$ be the set of non-neighbors of $a_i$ in $G[U]$. Note that $A_i\cup\{a_{i-2},a_{i+2}\}$ is a clique, and so  $|A_i|\le 2t$ since $G$ does not contain any $K_{2t+2}$-immersion.
Also note that since $G$ contains no stable set of size three, and hence $A_i\cap A_{i+2}=\emptyset$ for every $i$. 

As discussed above, we wish to find sets $X_1,X_3$ satisfying Claim \ref{claim:2.5X}. One might hope to choose $X_1:=A_3\cap Q$ and $X_3:=A_1\cap Q$; these sets indeed satisfy \ref{enum:2.5.1} and \ref{enum:2.5.3} but may fail to meet \ref{enum:2.5.2} in the case either $|A_1|$ or $|A_3|$ is small. This problem can be avoided by enlarging $A_1$ and $A_3$. This leads to the following definition of $A_1'...,A_5'$.

Let $A_1',...,A_5'$ be subsets of $U$ such that $\sum_{i=1}^{5}|A'_i|$ is as large as possible, and
\begin{equation}
\label{eq:2.5}
\left\{\begin{array}{l}
A_i\subseteq A_i',\\
|A_i'|\le 2t,\\
A_i'\cap A'_{i+2}=\emptyset,
\end{array}\right.
\ \ \ \forall 1\le i\le 5.
\end{equation}

\begin{claim}\label{claim:sum2t}
There exists $i$ such that $|A_i'|=|A_{i+2}'|=2t$.
\end{claim}

\begin{cproof}
Assume the claim is false.  
Then there exists $j$ such that $|A'_{j}|,|A'_{j+1}|,|A'_{j+2}|<2t$.
Without loss of generality, assume $|A_1'|,|A'_2|,|A'_3|<2t$.  

For every $i$, let $B_{i,i+1}=A'_i\cap A'_{i+1}$, and $D_i=A'_i\backslash (A_{i-1}\cup A'_{i+1})$. Then all 10 sets $D_i, B_{i,i+1}$ are pairwise disjoint, and $A'_i=B_{i-1,i}\cup D_i \cup B_{i,i+1}$. 
Note also that $D_i\cap A'_{i+1}=\emptyset$ and $D_i\cap A'_{i-1}=\emptyset$ for every $i$.

Suppose that there exists $v\in U$ such that $v\notin \bigcup_{i=1}^{5}A'_i$. Then $A_1'\cup\{v\},A_2',...,A_5'$  satisfy (\ref{eq:2.5}), while the sum of their cardinalities is larger, a contradiction. 
This gives $\bigcup_{i=1}^{5}A'_i=U$. In other words, 
$$\Big(\bigcup_{i=1}^{5}D_i\Big)\cup \Big(\bigcup_{i=1}^{5}B_{i,i+1}\Big)=U.$$ Since all these sets are pairwise disjoint, we have
\begin{equation}\label{eq:2.5.2}
\sum_{i=1}^{5}|D_i|+\sum_{i=1}^{5}|B_{i,i+1}|=|U|\ge 5t.
\end{equation}

Observe that if $|A'_i|<2t$ and there exists $v\in D_{i-1}\cup D_{i+1}$, then $(A'_i\cup\{v\})\cap A'_{i+2}=\emptyset$ and  $(A'_i\cup\{v\})\cap A'_{i-2}=\emptyset$. Hence
$A'_i\cup\{v\},A'_{i+1},...,A'_{i+4}$ satisfy (\ref{eq:2.5}), violating our choice to maximize the sum of their cardinalities. Hence if $|A_i'|<2t$, then $D_{i-1}=\emptyset$ and $D_{i+1}=\emptyset$.

Recall the assumption that $|A'_1|,|A'_2|,|A'_3|<2t$.
By the observation in the previous paragraph, we have $D_j=\emptyset$ for every $j$. Hence  from (\ref{eq:2.5.2}) we have $\sum_{i=1}^{5}|B_{i,i+1}|\ge 5t$. Also note that $|A_i'|=|B_{i,i-1}|+|D_i|+|B_{i,i+1}|=|B_{i,i-1}|+|B_{i,i+1}|$ for every $i$. This gives
$$10t\le2\sum_{i=1}^{5}|B_{i,i+1}|=\sum_{i=1}^{5}|A'_i|<10t,$$
a contradiction. This proves the claim. 
\end{cproof} 

Without loss of generality, we may suppose that $|A'_1|=|A'_3|=2t$.
\begin{claim}\label{cl:2.5s}
Let $X_1=A_3'\cap Q$ and $X_3=A_1'\cap Q$. Then $X_1,X_3$ satisfy conditions in Claim \ref{claim:2.5X}. 
\end{claim}
\begin{cproof}
We first show that \ref{enum:2.5.3} holds for $X_3$.
Recall that $A_3\subseteq A_3'$, $A_4\subseteq A_4'$ and 
$A_1'\cap (A_3'\cup A_4')=\emptyset$. Then $A_1'\cap (A_3\cup A_4)=\emptyset$, and so $X_3\cap (A_3\cup A_4)=\emptyset$ since $X_3\subseteq A_1'$. 
Hence for every $v\in X_3$, we have $va_3\in E$ and $va_4\in E$ (otherwise, $G$ contains a stable set of size three). 
Note that $B_{5,1}\cap B_{1,2}=\emptyset$. Hence for every $v\in X_3$, either $v\notin B_{5,1}$ or $v\notin B_{1,2}$. 
If $v\notin B_{5,1}$ then $v\notin A_5'$ (since $v\in A'_1$), and so $v \notin A_5$. This means that $v$ is adjacent to $a_5$. Otherwise, $v\notin B_{1,2}$, and by the same argument, $v$ is adjacent to $a_2$.
Hence, \ref{enum:2.5.3} holds for $X_3$. 

We now show that \ref{enum:2.5.2} holds for $X_3$. Let $M_1'=A_1'\cap M$ and $M_3'=A_3'\cap M$. Then by (\ref{eq:2.5}), we have $M_1\subseteq M_1'$, $M_3\subseteq M_3'$, and $M_1'\cap M_3'=\emptyset$. Besides, $X_3\cap M_1'\subseteq Q\cap M=\emptyset$ and $$X_3\cup M_1'=(A_1'\cap Q)\cup(A_1'\cap M)=A_1'\cap U=A_1'.$$
This gives $|X_3|+|M_1'|=|A_1'|=2t$, and so
$$|X_3|=2t-|M_1'|= |M|-|M_1'|=|M\backslash M_1'|\ge |M_3'|\ge |M_3|.$$
Hence \ref{enum:2.5.2} holds for $X_3$. 

By the same arguments, \ref{enum:2.5.2} and \ref{enum:2.5.1} hold for $X_1$. This proves the claim.
\end{cproof}
Claims \ref{claim:2.5X} and \ref{cl:2.5s} complete the proof of Theorem \ref{theorem:2.5}.
\end{proof}


\begin{thebibliography}{10}\label{bibliography}

\bibitem{AL} F. Abu-Khzam and M. Langston, 
\emph{Graph coloring and the immersion order}, 
Lecture Notes in Computer Science \textbf{2697} (2003), 394--403.

\bibitem{mohar} M. Devos, Z. Dvo\v{r}\'ak, J. Fox, J. McDonald, B. Mohar, and D. Scheide, 
\emph{Minimum degree condition forcing complete graph immersion}, Combinatorica \textbf{34} (2014), 279--298.

\bibitem{dvo} Z. Dvo\v{r}\'ak  and L. Yepremyan, 
\emph{Comptete graph immersions and minimum degree}, preprint, \href{http://arxiv.org/abs/1512.00513}{arXiv:1512.00513}.

\bibitem{DKMO10} M. Devos, K. Kawarabayashi, B. Mohar, and H. Okamura, 
\emph{Immersing small complete graphs}, Ars Math.
Contemp. \textbf{3} (2010), 139--146.

\bibitem{DM} P. Duchet and H. Meyniel, \textit{On Hadwiger’s number and the stability number}, Ann. Discrete Math. \textbf{13} (1982), 71--74.


\bibitem{FW16} J. Fox  and F. Wei, 
\emph{On the number of cliques in graphs with a forbidden subdivision or immersion}, preprint, \href{https://arxiv.org/abs/1606.06810}{arXiv:1606.06810}.


\bibitem{Had} H. Hadwiger,  \emph{Uber eine Klassifikation der Streckenkomplexe},
Vierteljschr. Naturforsch. Ges. Zurich {\bf 88} (1943), 133--143.

\bibitem{Ko} A. Kostochka,  \textit{ Lower bound of the Hadwiger number of graphs by their average degree},
Combinatorica \textbf{4} (1984), 307--316.

\bibitem{LW} T.-N. Le and P. Wollan, {\it Forcing clique immersions through chromatic number}, Electronic Notes Disc.~Math.~{\bf 54(1)} (2016), 121--126.

\bibitem{LM} F. Lescure and H. Meynial, {\it On a problem upon configurations contained in graphs
with given chromatic number}, Ann. Discrete Math. {\bf 41} (1989), 325--331.
Graph theory in memory of G.A. Dirac, Sandbjerg (1985).

\bibitem{PST} M. D. Plummer, M. Stiebitz, and B. Toft, \textit{On a special case of Hadwiger’s conjecture}.
Discuss. Math. Graph Theory
\textbf{23} (2005), 333--363.

\bibitem{RS} N. Robertson, P.D. Seymour, \textit{Graph minors XXIII, Nash-Williams’ immersion conjecture}, J. Combin. Theory Ser. B \textbf{100} (2010), 181--205.

\bibitem{RST} N. Robertson, P.D. Seymour, R. Thomas, \textit{Hadwiger’s conjecture for K6-free graphs}, Combinatorica
\textbf{13} (1993), 279--361.

\bibitem{Tho} A. Thomason, \textit{An extremal function for contractions of graphs},  Math. Proc. Cambridge Philos. Soc.
\textbf{95} (1984), 261--265.

\bibitem{V17} S. Vergara, \textit{Complete graph immersions in dense graphs}, Discrete Math. \textbf{340} (5) (2017), 1019--1027.

\bibitem{wa} K. Wagner,  \textit{\"Uber  eine  Eigenschaft  der  ebenen
Komplexe}, Math. Ann. \textbf{114} (1937), 570--590.


\bibitem{Wo} P. Wollan, \textit{The structure of graphs not admitting a fixed immersion}, J. Combin. Theory Ser. B \textbf{110} (2015), 47--66.
\end{thebibliography}
\end{document}